%% file: V1.tex
\documentclass[pdf]{article}

\usepackage{pdfpages}
\usepackage{hyperref}
\usepackage{xcolor}
\usepackage[utf8]{inputenc}
\usepackage{amsmath}
\usepackage{amsfonts}
\usepackage{amssymb}
\usepackage{amsthm}
\usepackage{marginnote}
\usepackage{enumerate}
\usepackage{graphicx}
\usepackage[normalem]{ulem}
\usepackage{fancyhdr}

\include{commands}

\makeatother

\begin{document}

\title{Unimodular Random Measured Metric Spaces and Palm Theory on Them}

\author{Ali Khezeli \footnote{Inria Paris, ali.khezeli@inria.fr} \footnote{Department of Applied Mathematics, Faculty of Mathematical Sciences, Tarbiat Modares University, P.O. Box 14115-134, Tehran, Iran, khezeli@modares.ac.ir}}

\maketitle


\begin{abstract}
	In this work, we define the notion of unimodular random measured metric spaces as a common generalization of various other notions. This includes the discrete cases like unimodular graphs and stationary point processes, as well as the non-discrete cases like stationary random measures and the continuum metric spaces arising as scaling limits of graphs. We provide various examples and prove many general results; e.g., on weak limits, re-rooting invariance, random walks, ergodic decomposition, amenability and balancing transport kernels. In addition, we generalize the Palm theory to point processes and random measures on a given unimodular space. This is useful for Palm calculations and also for reducing some problems to the discrete cases.
\end{abstract}

\section{Introduction}

The \textit{mass transport principle (MTP)} refers to some key equations that appear in various forms in a number of subfields of probability theory and other fields of mathematics. These equations capture (and sometimes formalize) the intuitive notion of \textit{stochastic homogeneity}, which, in different contexts, appears as stationarity, unimodularity, typicality of a point, re-rooting invariance, involution invariance, or just invariance. 
The similarity of the different versions of the MTP results in fruitful connections between various fields. 

In Subsection~\ref{intro:MTP}, we provide a quick survey of the MTP in the literature. Then, the contributions of the present paper are introduced in the next parts of the introduction.



\subsection{The Mass Transport Principle in the Literature}
\label{intro:MTP}

The mass transport principle (MTP) was first used in~\cite{Ha97} for studying percolation on trees. It was developed further in~\cite{BLPS} for studying group-invariant percolation on graphs and in~\cite{BeSc01} for proving recurrence of local limits of planar graphs. \cite{objective} also established a property called \textit{involution invariance} for local limits of finite graphs\footnote{The MTP for local limits of graphs was also observed in~\cite{BeSc01}.}. This property turned out to be a special case of the MTP, and in fact, equivalent to it. Then, \cite{processes} defined \textit{unimodular random graphs} by simply using the MTP as the definition, provided various general examples, and established many properties for them. 
The term \textit{unimodular} was chosen in~\cite{processes} because, in order that a fixed transitive graph $G$ satisfies the MTP, it is necessary and sufficient that the automorphism group of $G$ is a unimodular group. 

It is useful to recall the MTP for unimodular graphs. A \textit{mass transport} or a \textit{transport function} is a function $g$ that takes as input a graph $G$ and two vertices of $G$ and outputs a value in $\mathbb R^{\geq 0}$. We think of $g(G,u,v)$ as the mass going from $u$ to $v$. The assumptions on $g$ are measurability (to be defined suitably) and \textit{invariance}, where the latter means here that $g$ is invariant under isomorphisms. 
A random rooted graph $[\bs G, \bs o]$ (where the graph $\bs G$ and the root vertex $\bs o$ are random) is called unimodular if it satisfies the following MTP for any transport function $g$:
\begin{eqnarray}
	\label{eq:MTP-graph}
	\omid{\sum_{x\in \bs G} g(\bs G, \bs o, x)} = \omid{\sum_{x\in \bs G} g(\bs G,x, \bs o)}.
\end{eqnarray}

Similar versions of the MTP exist for point processes in $\mathbb R^d$. For a recent application, one can mention the use of the MTP in constructing perfect matching between point processes and fair tessellations (see e.g., \cite{HoPe05} and~\cite{DeHoMa17}). 
These MTPs are special cases of much older theorems in stochastic geometry; e.g., Mecke's formula, Mecke's theorem and Neveu's exchange formula, which had been widely used in stochastic geometry. In particular, Mecke's theorem (see e.g., \cite{ThLa09}) can be rephrased as follows: If $\Phi_0$ is the Palm version of a stationary point process (i.e., conditioned to contain the origin, or equivaletnly, seeing the point process form a \textit{typical point}), then
\begin{eqnarray}
	\label{eq:MTP-pp}
	\omid{\sum_{x\in \Phi_0} g(\Phi_0, 0, x)} = \omid{\sum_{x\in \Phi_0} g(\Phi_0,x, 0)}.
\end{eqnarray}
This equation looks similar to~\eqref{eq:MTP-graph} with the difference that the random objects have different natures and the invariance condition for the transport function $g=g(\Phi,u,v)$ is the invariance under Euclidean translations. 
In fact,~\eqref{eq:MTP-pp} characterizes a larger class of point processes which are called \textit{point-stationary point processes} in~\cite{ThLa09} (e.g., the zero set or the graph of the simple random walk). See Subsection~\ref{subsec:pointprocess} for further discussion.

Due to the similarity of the MTP in (point-) stationary point processes and unimodular random graphs, a lot of connections have been observed between the two theories. In particular, \cite{processes} noted that any graph that is constructed on a stationary point process (with some equivariance and measurability condition) gives rise to a unimodular random graph. Also, various results and proof techniques in one theory can be \textit{translated} into the other one with minor modifications (see~\cite{I}). 
The previous work~\cite{I} unifies these two notions by defining \textit{unimodular random discrete spaces}.

The theorems mentioned above for stationary point processes are in fact more general and apply to stationary random measures. In particular, Mecke's theorem in the general form can be rephrased equivalently in a MTP form as follows: If $\Psi_0$ is the Palm version of a stationary random measure on $\mathbb R^d$, then
\begin{equation}
	\label{eq:MTP-measure}
	\omid{\int g(\Psi_0,0,x)d\Psi(x)} = \omid{\int g(\Psi_0,x,0)d\Psi(x)}
\end{equation}
for every function $g=g(\Psi_0,u,v)\geq 0$ that is translation-invariant and measurable. As in the previous case, \eqref{eq:MTP-measure} characterizes a larger class of random measures, which are called \textit{mass-stationary random measures} in~\cite{ThLa09}. See Subsection~\ref{subsec:pointprocess} for further discussion.


As mentioned above, the MTP \eqref{eq:MTP-graph} is satisfied for local limits of graphs. Similar properties have been established for some instances of scaling limits of graphs, where scaling limit means that the graph-distance metric is scaled by some small factor and the limit means the limit of a sequence of random metric spaces (sometime, the metric spaces are \textit{rooted}; i.e., have a distinguished point, and/or \textit{measured}; i.e., have a distinguished measure). Most notably, various instances that have a compact rooted measured scaling limit satisfy the so called \textit{re-rooting invariance property}; i.e., the distribution is invariant under changing the root  randomly (according to the distinguished measure on the model). This is the case for the \textit{Brownian continuum random tree} \cite{Al91crtII}, \textit{stable trees}~\cite{DuLe05} and the \textit{Brownian map}~\cite{Le19}. Also, in a few examples where the scaling limit is not compact, an MTP similar to~\eqref{eq:MTP-measure} has been observed (e.g., in~\cite{Bu18}). However, a general rigorous MTP result for scaling limits seems to be missing in the literature. This is done in this work by providing a general form of the MTP and by showing that it is preserved under weak limits. 
This general result readily implies the MTP and re-rooting invariance in the known examples of scaling limits. We will also provide some re-rooting invariance properties in the non-compact case as well.

In this occasion, we should also mention the mass transport principle in the theory of \textit{countable Borel equivalence relations}. This will be introduced in the next subsection and will be discussed further in Subsection~\ref{subsec:Borel}.

\subsection{Unification of the MTP by Unimodular \rmm Spaces}

In the previous subsection, we recalled that the theories of unimodular graphs, point processes, random measures and scaling limits exhibit the mass transport principle in various forms. The similarity of these formulas provide fruitful connections between these theories. So, it is natural to ask for a general theory of the MTP that unifies these notions. We saw that the discrete cases (unimodular graphs and point processes) can be unified by the notion of unimodular discrete spaces~\cite{I} (these cases are also tightly connected to the theory of countable Borel equivalence relations). 

The first main task of this work is providing a unification of the MTP in the discrete and non-discrete cases. 
The unification is provided by introducing \textit{unimodular random rooted measured metric spaces}, or in short, \textit{unimodular rmm spaces} in Section~\ref{sec:unimodular}. This can be thought of a common generalization of unimodular graphs, (point-) stationary point processes, (mass-) stationary random measures and the random continuum spaces arising in scaling limits. In this part, establishing the MTP for general scaling limits is novel and is one of the main motivations of this work.

Unimodular \rmm spaces are defined as random triples $[\bs X, \bs o, \bs \mu]$, where $\bs X$ is a \textit{boundedly-compact} metric space, $\bs o$ is a point of $\bs X$ called \textit{the root}, and $\bs \mu$ is a \textit{boundedly-finite} Borel measure on $\bs X$. The \textit{Gromov-Hausdorff-Prokhorov metric}, given in the most general form in~\cite{Kh19ghp}, provides the measure-theoretic requirements for this definition.
Then, the definition of unimodularity in this setting is given by modifying the MTP~\eqref{eq:MTP-measure} accordingly. Note that a distinguished measure is needed on $\bs X$ in order that the MTP makes sense. Many specific examples and general categories of unimodular \rmm spaces are discussed in Section~\ref{sec:examples-general}.

In Sections~\ref{sec:basic} and~\ref{sec:further}, we provide general results on unimodular spaces. In particular, we prove that unimodularity is preserved under weak convergence. This implies that all (measured) scaling limits are unimodular (which is one of the main motivations of this work), and hence, in the compact case satisfy the re-rooting invariance property. We also study re-rooting in the non-compact case and the properties of random walks on a unimodular \rmm space. We define ergodicity and prove an ergodic decomposition theorem. Also, various definitions of amenability are provided and proved to be equivalent, based on similar definitions in other fields of mathematics; e.g., definition by local means, approximate means, hyperfinitenss and Folner condition. It is also proved that the number of topological ends of a unimodular \rmm space belongs to $\{0,1,2,\infty\}$ almost surely.

We should mention that the discrete cases of the MTP, discussed above, are connected to the theory of \textit{countable Borel equivalence relations} (CBER), which has totally different roots (ergodic theory of dynamical systems, group actions and orbit equivalence theory)\footnote{This theory is almost as old as Mecke's theorem and is very useful in probability theory, but still deserves to be more recognized in the probability community. Some results in this theory have been re-invented later by probabilists; e.g., the existence of \textit{one-ended treeing} in the amenable case (see e.g., \cite{CoFeWi81}, \cite{HoPe03} and~\cite{Ti04}) and  factor point processes (see~\cite{KhMe23}).}. There exists a MTP-like formula in this theory, which in fact defines when an equivalence relation is measure-preserving; i.e., the measure is \textit{invariant}. It is a generalization of the measure-preserving property for dynamical systems and group actions. 
As discussed in~\cite{processes}, this notion is connected to unimodular graphs via \textit{graphings}.
While the two theories have substantial overlap, their view point and motivations are different. 
In fact, the theory of CBERs can be thought of as the mathematical ground under the theory of unimodular graphs; just like the distinction of measure theory and probability theory.
The unification of the MTP in this work does not cover CBERs since we focus on random metric spaces. We will still use this theory in proving some of the results. This will be discussed further in Subsection~\ref{subsec:Borel}.

\subsection{Generalization of Palm Theory}

In this work, we also consider random measures and point processes on a given unimodular \rmm space.
In this view, the base space can be thought of a generalization of the Euclidean or hyperbolic spaces.
Heuristically, the \textit{intensity} of a point process $\Phi$ on $[\bs X, \bs o, \bs \mu]$ is the \textit{expected number of points of $\Phi$ per unit measure} (measured by $\bs \mu$). Also, the \textit{mean} of some quantity on the points of $\Phi$ is the expectation of that quantity at a \textit{typical point} of $\Phi$. These notions will be formalized by generalizing the Palm theory in Section~\ref{sec:Palm}, which is the second main task of this work. We will generalize various notions and theorems in stochastic geometry; e.g., the Campbell measure, Pam distribution, the Campbell formula, Mecke's theorem, Neveu's exchange formula and the existence of balancing transport kernels. We also show that the Palm theory generalizes Palm inversion at the same time (i.e., reconstructing the probability measure from the Palm distribution) and also generalizes various constructions of unimodular graphs by adding vertices and edges to another given unimodular graph (e.g., the dual of a unimodular planar graph).

An important application of the Palm theory in this work is to construct a countable Borel equivalence relation equipped with an invariant measure. This is done by means of adding a Poisson point process and considering its Palm version and is used to prove some of the theorems (e.g., amenability and invariant disintegration) using the results of the theory of countable Borel equivalence relations. These applications are included in Section~\ref{sec:Palm-application}.

\section{Unimodular \rmm Spaces}
\label{sec:unimodular}

In this section, we define \textit{unimodular \rmm spaces} as a common generalization of unimodular graphs, stationary point processes, stationary random measures, etc. The definition is based on a mass transport principle similar to~\eqref{eq:MTP-graph}, \eqref{eq:MTP-pp} and~\eqref{eq:MTP-measure}. Note that in the MTP, a distinguished point and a measures should be presumed (since the spaces are not necessarily discrete). So we start by discussing rooted measured metric spaces and some notation.

\subsection{Notation}

If $X$ is a metric space, the 
closed ball of radius $r$ centered at $x\in X$ is denoted by $\cball{r}{x}:=\{y\in X: d(x,y)\leq r\}$. The open ball is denoted by $\oball{r}{x}$. Also, all measures are assumed to be Borel measures. 
If $f:X\to\mathbb R^{\geq 0}$ is measurable, the measure $f\mu$ on $X$ is defined by $(f\mu)(A):=\int_A fd\mu$. If $\mu$ is a probability measure, \textit{biasing} $\mu$ by $f$ means considering the probability measure $\frac 1 c (f\mu)$, where $c=\int_X fd\mu$. The Prokhorov metric between finite measures on $X$ is denoted by $\prokhorov$.


\subsection{The Space of Rooted Measured Metric Spaces}
\label{subsec:M_*}

A \defstyle{rooted measured metric space}, abbreviated by a \defstyle{\rmm space}, is a tuple $(X,o,\mu)$ where $X$ is a metric space, 
$\mu$ is a non-negative Borel measure on $X$ and $o$ is a distinguished point of $X$ called the \textbf{root}. For simplicity of notation, the metric on $X$ is always denoted by $d$ if there is no ambiguity, otherwise, it will be denoted by $d^X$. 
In this paper, we always assume that the metric space is \defstyle{boundedly compact} (i.e. every closed ball is compact) and $\mu$ is \defstyle{boundedly finite} (i.e. every ball has finite measure under $\mu$).  
Similarly, a \defstyle{doubly-rooted} measured metric space is a tuple $(X,o_1,o_2,\mu)$, where $\mu$ is as before and $o_1$ and $o_2$ are two ordered distinguished points of $X$. 
Note that we do not require that $X$ is a geodesic space. It can be even disconnected.

An \defstyle{isomorphism} (or a GHP-isometry) between two \rmm spaces $(X,o,\mu)$ and $(X',o',\mu')$ is a bijective isometry $\rho:X\rightarrow X'$ such that 
$\rho(o)=o'$ and $\rho_*\mu=\mu'$, where $\rho_* (\mu)$ represents the push forward of the measure $\mu$ under $\rho$. If such $\rho$ exists, then $(X,o,\mu)$ and $(X',o',\mu')$ are called \defstyle{isomorphic}. 
Isomorphisms between doubly-rooted spaces are defined similarly. Under this equivalence relation, the equivalence class containing $(X,o,\mu)$ (resp. $(X,o_1,o_2,\mu)$) is denoted by $[X,o,\mu]$(resp. $[X,o_1,o_2,\mu]$). 

Let $\mstar$ be the set of equivalence classes of \rmm spaces under isomorphisms. Similarly, define $\mdoublestar$ for doubly-rooted measured metric spaces. These sets become Polish spaces under the \defstyle{Gromov-Hausdorff-Prokhorov (GHP) metric} (see~\cite{Kh19ghp} for the general case of the GHP metric and its history). To keep focus on the main goals, we skip the definition of the metric and we just mention the GHP topology (see Section~3 of~\cite{Kh19generalization}): A sequence $[X_n,o_n,\mu_n]$ converges to $[X,o,\mu]$ when these spaces can be embedded isometrically in a common boundedly-compact metric space $Z$ such that, after the embedding, $o_n$ converges to $o$, $X_n$ converges to $X$ as closed subsets of $Z$ (with the Fell topology) and $\mu_n$ converges to $\mu$ in the vague topology (convergence against compactly-supported continuous functions).

Polishness of $\mstar$ allows one to use classical tools in probability theory for studying random elements of $\mstar$, which are called \defstyle{random \rmm spaces} here. In particular, every random rooted graph defines a random \rmm space (equipped with the graph-distance metric and the counting measure on the vertices). The same is true for point processes. Also, $\mathbb R^d$, rooted at the origin and equipped with a random measures on $\mathbb R^d$, defines a random \rmm space. The condition of boundedly-compactness matches the conditions of locally-finiteness in each example.

\begin{convention}
	A random \rmm spaces is shown by a tuple of bold symbols like $[\bs X, \bs o, \bs \mu]$. By convention, we will use the familiar symbols $\mathbb P$ and $\mathbb E$ for most random objects (instead of writing them in the form of long integrals) even if they live in different spaces and even if extra randomness is introduced. The following explanation helps to reduce the possible confusions. 
	Note that the tuple determines just one random element of $\mstar$ and the three symbols $\bs X, \bs o$ and $\bs \mu$ are meaningless separately. Any formula containing these symbols should be well defined for an isomorphism class of \rmm spaces. By contrast, in Section~\ref{sec:Palm}, we will consider more than one probability measure on $\mstar$ or similar spaces. In this case, one may also think of $[\bs X, \bs o, \bs \mu]$ as a symbol for a generic element of $\mstar$ and use formulas like $\myprob{[\bs X, \bs o, \bs \mu]\in A}$ and $Q([\bs X,\bs o, \bs \mu]\in A)$ for different probability measures $\mathbb P$ and $Q$.
\end{convention}


\subsection{Unimodularity}
\label{subsec:unimodular}

In this subsection, we define unimodular \rmm spaces and a few basic examples.

\begin{definition}
	A \defstyle{transport function} is a measurable function $g:\mdoublestar\to\mathbb R^{\geq 0}$. The value $g(X,u,v,\mu)$ is interpreted as the mass that $u$ sends to $v$ and is also abbreviated by $g(u,v)$ if $X$ and $\mu$ are understood. Let $g^+(u):=\int g(u,x) d\mu (x)$ and $g^-(u):= \int g(x,u)d\mu(x)$ represent the \defstyle{outgoing mass} from $u$ and the \defstyle{incoming mass} to $u$ respectively.
\end{definition}

\begin{definition}
	A random \rmm space is called \defstyle{unimodular} if the following mass transport principle holds: For every transport function $g$,
	\begin{equation}
		\label{eq:unimodular}
		\omid{\int_{\bs X} g(\bs o, x) d\bs \mu (x)} = \omid{\int_{\bs X} g(x, \bs o) d\bs \mu (x)}.
	\end{equation}
	It is called \defstyle{nontrivial} if $\bs\mu\neq 0$ a.s. and \defstyle{proper} if $\mathrm{supp}(\bs\mu)=\bs X$ a.s.
\end{definition}

In words, the expected outgoing mass from the root should be equal to the expected incoming mass to the root.
The case $\bs\mu:=0$ trivially satisfies the MTP and is not very interesting. However, there are some interesting classes of non-proper examples (e.g., when extending a unimodular graph, Example~\ref{ex:extension}, or $\mathbb R^d$ or any other space equipped with a point process or a random measure, Example~\ref{ex:PPandRM} and Definition~\ref{def:equivmeasure}). Another basic example is when $[\bs X, \bs o]$ is arbitrary and $\bs \mu:=\delta_{\bs o}$ is the Dirac measure at $\bs o$. More generally, finite measures provide basic examples explained below.

\begin{example}[Finite Measure]
	\label{ex:finitemeasure}
	When $\bs \mu$ is a finite measure a.s., unimodularity is equivalent to \defstyle{re-rooting invariance}: If $\bs o'$ is an additional random point of $\bs X$ chosen with distribution proportional to $\bs\mu$, then $[\bs X, \bs o', \bs\mu]$ has the same distribution as $[\bs X, \bs o ,\bs \mu]$ (see Theorem~\ref{thm:rerooting}). Loosely speaking, the root is a random point of $\bs X$ chosen with distribution proportional to $\bs \mu$.
	%
\end{example}

Although the infinite case is more interesting for our purpose, the finite case appears in many important examples of scaling limits when the limiting space is compact (see Subsection~\ref{subsec:scalinglimit}). We will see that the re-rooting invariance property in these examples is a quick corollary of the fact that weak convergence preserves unimodularity (Lemma~\ref{lem:weaklimit}).

\begin{example}[Product]
	\label{ex:product}
	Let $[\bs X_i,\bs o_i,\bs \mu_i]$ be unimodular for $i=1,2$. Then, $[\bs X_1\times\bs X_2, (\bs o_1,\bs o_2), \bs \mu_1\otimes \bs \mu_2]$ is also unimodular. Here, the metric on $\bs X_1\times \bs X_2$ can be the max-metric or the sum-metric.
\end{example}

\begin{example}[Biasing]
	\label{ex:biasing}
	Let $[\bs X, \bs o, \bs \mu]$ be unimodular and $b:\mstar\to\mathbb R^{\geq 0}$ be a measurable function such that $\omid{b(\bs o)}=1$. Let $b\bs\mu$ be the measure on $\bs X$ defined by $b\bs\mu(A):=\int_A b(x)d\bs\mu(x)$. Let $[\bs X',\bs o', \bs \mu']$ be obtained by changing the measure $\bs \mu$ to $b\bs\mu$ and then by biasing the probability measure by $b(\bs o)$; i.e.,
	\[
	\omid{h(\bs X', \bs o', \bs \mu')} = \omid{b(\bs X, \bs o, \bs \mu) h(\bs X, \bs o, b\bs\mu)}.
	\]
	Then, $[\bs X', \bs o', \bs \mu']$ is a unimodular \rmm space. This can be seen by verifying the MTP directly. In fact, in Example~\ref{ex:biasing-Palm}, this will be shown to be the Palm version of $b\bs \mu$ regarded as an additional measure on $[\bs X, \bs o, \bs \mu]$.
\end{example}


A heuristic interpretation of unimodularity is that the root is a \textit{typical point} of the measure $\bs \mu$. When $\bs \mu$ is a finite measure, this heuristic is rigorous according to Example~\ref{ex:finitemeasure}. In the general case, the heuristic is that the expectation of any (equivariant) quantity evaluated at the root is an interpretation of the average of that quantity over the points of the metric space. More precisely, for measurable functions $f(\bs o):=f(\bs X, \bs o, \bs \mu)$, the expectation $\omid{f(\bs o)}$ is an interpretation of the average of $f(\bs X, \cdot, \bs \mu)$ over the points of $\bs X$, where the average is taken according to the measure $\bs\mu$ (in fact, this should be modified in the \textit{non-ergodic} case; see Subsection~\ref{subsec:ergodicDecomposition}). See Corollary~\ref{cor:average} for a rigorous statement.


\section{Basic Properties of Unimodularity}
\label{sec:basic}
In this section, we will provide basic properties of unimodularity, which are useful in the discussion of the examples in the next section. Further properties will be provided in Section~\ref{sec:further}.

\subsection{Weak Limits}

\begin{lemma}[Weak Limits]
	\label{lem:weaklimit}
	Unimodularity is preserved under weak limits.
\end{lemma}

This is similar to the case of unimodular graphs, but the proof is more involved since a convergent sequence of \rmm spaces does not necessarily \textit{stabilize} in a given window. This will be handled by a generalization of Strassen's theorem.

Lemma~\ref{lem:weaklimit} and Example~\ref{ex:finitemeasure} give naturally the following generalization of \textit{soficity}:


\begin{problem}
	Is every unimodular \rmm space \defstyle{sofic}? i.e., is it the weak limit of a sequence of deterministic compact measured metric spaces $(X_n,\mu_n)$ rooted at a random point with distribution proportional to $\mu_n$?
\end{problem}

This generalizes Question~10.1 of~\cite{processes}, which involves unimodular graphs. Note that every such weak limit is unimodular. Also, one can assume that $X_n$ is a finite metric space and $\mu_n$ is a multiple of the counting measure without loss of generality.

\begin{proof}[Proof of Lemma~\ref{lem:weaklimit}]
	Assume $[\bs X_n, \bs o_n,\bs \mu_n]$ is unimodular ($n=1,2,\ldots$) and converges to $[\bs X, \bs o,\bs \mu]$. To prove unimodularity of $[\bs X, \bs o, \bs \mu]$, it is enough to prove the MTP~\eqref{eq:unimodular} for bounded continuous functions $g$.
	Define $f:\mstar\to\mathbb R^{\geq 0}$ by
	\[
		f(X,o,\mu):=\int_X g(X,o,p,\mu)d\mu(p).
	\]
	By approximating $g$ by simpler functions, it is enough to assume that for some $R<\infty$, $g(X,o,p,\mu)=0$ whenever $d(o,p)\geq R$ and also $f$ and $g$ are bounded by $R$. In this case, we will prove in the next paragraph that $f$ is a bounded continuous function. So, weak convergence implies that $\omid{f(\bs X_n,\bs o_n,\bs \mu_n)}\to\omid{f(\bs X, \bs o, \bs \mu)}$. A similar argument holds when $o$ and $p$ are swapped in the definition of $f$. Now, the MTP for $[\bs X, \bs o, \bs \mu]$ is implied by taking limit from the MTP for $[\bs X_n,\bs o_n,\bs \mu_n]$ and claim is proved.
	
	The rest is devoted to proving the continuity of $f$ and can be skipped at first reading.
	Assume $[X_n,o_n,\mu_n]$ are deterministic \rmm spaces converging to $[X,o,\mu]$. One can assume that $X_n$'s are subspaces of a common boundedly-compact metric space $Z$ converging to $X\subseteq Z$, $o_n\to o$ and $\mu_n$ converges vaguely to $\mu$. So there exists $\epsilon_n\geq 0$ and measures $\mu'_n$ sandwiched between the restrictions of $\mu_n$ to $\cball{R+3-\epsilon_n}{o}$ and $\cball{R+3+\epsilon_n}{o}$  (the balls are in $Z$) such that, if $\mu'$ is the restriction of $\mu$ to $\cball{R+3}{o}$, then $d_P(\mu'_n,\mu')\leq \epsilon_n$ (see Section~3 of~\cite{Kh19generalization}). We may assume $\epsilon_n<1$. By the generalized Strassen theorem (Theorem~2.1 of~\cite{Kh19ghp}), there exists an \textit{approximate coupling} of $\mu'_n$ and $\mu'$; i.e., a measure $\alpha_n$ on $X_n\times X$ such that
	\[
		\norm{\pi_{1*}\alpha_n-\mu'_n} + \norm{\pi_{2*}\alpha_n-\mu'} + \alpha_n(\{(x,y)\in X_n\times X: d^Z(x,y)> \epsilon_n\})\leq\epsilon_n.
	\]
	Here, $\pi_1$ and $\pi_2$ are the projections from $X_n\times X$ to $X_n$ and $X$ respectively.
	Also, by the assumptions on $g$, the convergence and a compactness argument, there exists $M<\infty$ such that for all $n$,  $\sup\{g(X_n,o_n,p,\mu_n):p\in X_n\}\leq M$ and the same holds for $X$. In addition, by considering the modulus of continuity of $g$ on a suitable compact set, one finds that $\delta_n:=\sup\{\norm{g(X_n,o_n,p,\mu_n)-g(X,o,q,\mu)}: d^Z(p,q)\leq\epsilon_n\}\to 0$.  Now,	
	\begin{eqnarray*}
		&& 
		\norm{f(X_n,o_n,\mu_n)-f(X,o,\mu)}\\ 
		&=& \norm{\int_{X_n}g(X_n,o_n,p,\mu_n)d\mu_n(p) - \int_X g(X,o,q,\mu)d\mu(q)}\\
		&\leq& M\norm{\pi_{1*}\alpha_n-\mu'_n} + M\norm{\pi_{2*}\alpha_n-\mu'} +\\
		&& \norm{\int_{X_n}g(X_n,o_n,p,\mu_n)d(\pi_{1*}\alpha_n)(p) - \int_X g(X,o,q,\mu)d(\pi_{2*}\alpha_n)(q)}\\
		&\leq& M\norm{\pi_{1*}\alpha_n-\mu'_n} + M\norm{\pi_{2*}\alpha_n-\mu'} +\\
		&& \int\int \norm{g(X_n,o_n,p,\mu_n)- g(X,o,q,\mu)} d\alpha_n(p,q)\\
		&\leq& M\big(\norm{\pi_{1*}\alpha_n-\mu'_n} + \norm{\pi_{2*}\alpha_n-\mu'} +  \alpha_n(\{(x,y)\in X_n\times X: d(x,y)> \epsilon_n\})\big) \\
		&& +\int\int \norm{g(X_n,o_n,p,\mu_n)- g(X,o,q,\mu)}\identity{\{d(p,q)\leq\epsilon_n\}} d\alpha_n(p,q)\\
		&\leq& M\epsilon_n + \delta_n \norm{\alpha_n}\\
		&\leq& M\epsilon_n + \delta_n\epsilon_n + \delta_n \norm{\mu'}.
	\end{eqnarray*}
	It follows that $\norm{f(X_n,o_n,\mu_n)-f(X,o,\mu)}\to 0$ and the continuity of $f$ is proved. This finishes the proof of the lemma.
\end{proof}

\subsection{Subset Selection}


Unimodularity means heuristically that the root is a \textit{typical point}. In particular, every property of the points that has zero chance to be seen at the root, is observed at almost no other point. This is formalized in the following easy but important lemma.

\begin{definition}
	\label{def:subset}
	Let $[\bs X, \bs o, \bs \mu]$ be a random \rmm space and $A\subseteq \mstar$ be measurable. The set $\bs S:=\bs S(\bs X, \bs \mu):=\{p\in\bs X: (\bs X, p,\bs \mu)\in A\}$ is called a \defstyle{factor subset} of $\bs X$. Note that the factor subset is a function of $\bs X$ and $\bs\mu$ and does not depend on $\bs o$.
	
\end{definition}

\begin{lemma}[Everything Happens At Root]
	\label{lem:happensatroot}
	Let $[\bs X, \bs o, \bs \mu]$ be a nontrivial unimodular \rmm space. For every factor subset $\bs S$,
	\begin{eqnarray*}
		\bs o\in\bs S \text{ a.s. } &\iff & \bs S \text{ has full measure w.r.t. } \bs \mu, \text{ a.s.},\\
		\myprob{\bs o\in\bs S}>0 &\iff& \myprob{\bs\mu(\bs S)>0}>0.
	\end{eqnarray*}
\end{lemma}

\begin{proof}
	It is enough to prove the first claim. Define $g(u,v):=\identity{\{v\not\in \bs S\}}$. Then, $g^+(\bs o)=\bs\mu(\bs X\setminus \bs S)$ and $g^-(\bs o) = \bs\mu(\bs X)\identity{\{\bs o\not\in \bs S\}}$. Since $\bs \mu(\bs X)>0$ a.s., the claim follows by the MTP~\eqref{eq:unimodular}.
\end{proof}

This is a generalization of Lemma~2.3 of~\cite{processes}. It can also be generalized by allowing extra randomness as well, but some care is needed that will be discussed in Subsection~\ref{subsec:extra}.

By letting $\bs S:=\supp(\bs\mu)$, one immediately obtains:
\begin{corollary}
	\label{cor:support}
	If $[\bs X, \bs o,\bs \mu]$ is a nontrivial unimodular \rmm space, then $\bs o\in \supp(\bs{\mu})$ a.s. 
\end{corollary}

\begin{remark}
	Note that, unlike the discrete cases, one cannot replace the last statement in Lemma~\ref{lem:happensatroot} with $\myprob{\bs S\neq \emptyset}>0$. In the language of Borel equivalence relations, despite the case of countable equivalence relations, the saturation of a null set can have positive measure. 
\end{remark}

\begin{lemma}[Bounded Selection]
	\label{lem:subset}
	If $\bs{\mu}(\bs X)=\infty$ a.s., then every factor subset $\bs S$ of $\bs X$ is either empty or unbounded a.s. Also, $\bs \mu(\bs S)\in\{0,\infty\}$ a.s.
\end{lemma}

This generalizes Corollary~2.10 of~\cite{eft} for unimodular graphs.  As mentioned before Proposition~4 of~\cite{Lo20}, this is related to Poincar\'e's recurrence theorem. See also Lemma~\ref{lem:ends} for a further generalization.

\begin{proof}
	Let $A$ be an event in which $\bs S$ is nonempty and bounded. Hence, $\bs \mu(\bs S)<\infty$ on $A$. By replacing it with a suitable neighborhood of $\bs S$ if necessary, one may assume $\bs \mu(\bs S)>0$ on $A$. So it is enough to prove the second claim. Let $B$ be the event $0<\bs \mu(\bs S)<\infty$. Let $g(u,v):=(\bs \mu(\bs S))^{-1} \identity{\{v\in\bs S\}}\identity{B}$. Then, $g^+(\bs o)=\identity{B}$ and $g^-(\bs o)=\infty$ if $\bs o\in\bs S$ and $B$ holds. If $\myprob{B}>0$, then the latter holds with positive probability by Lemma~\ref{lem:happensatroot}. This contradicts the MTP for $g$.
\end{proof}

\subsection{Re-rooting Invariance}
\label{subsec:rerooting}

In the following proposition, re-rooting a unimodular \rmm space is considered even in the non-compact case. Assume for each \rmm space $(X,o,\mu)$ a probability measure $k_o=k_{(X,o,\mu)}$ is given on $X$. This will be considered as the law for changing the root from $o$ to a new root. By letting $(X,\mu)$ fixed and letting $o$ vary, this can be regarded as a Markovian kernel $k^{(X,\mu)}$ on $X$ (assuming the following measurability condition). It is called an \defstyle{equivariant Markovian kernel} if it is invariant under the isomorphisms of \rmm spaces and for each measurable set $A\subseteq \mdoublestar$, the function $[X,o,\mu] \mapsto k_{o}(\{y\in X: [X,o,y,\mu]\in A \})$ is a measurable function on $\mstar$. 

Given a deterministic pair $(X,\mu)$, an equivariant Markovian kernel $k$ transports $\mu$ to another measure on $X$ defined by $\mu'(\cdot):= \int_X k_y(\cdot) d\mu(y)$. If $\mu'=\mu$, then $\mu$ is called a \defstyle{stationary measure} for the kernel on $X$.

Let $[\bs X, \bs o, \bs \mu]$ be a unimodular \rmm space and $k$ be an equivariant Markovian kernel. Conditional to $[\bs X, \bs o,  \bs \mu]$, choose $\bs o'\in\bs X$ randomly with distribution $k_{\bs o}(\cdot)$, which is regarded as a new root. 

\begin{theorem}[Re-Rooting Invariance]
	\label{thm:rerooting}
	Assume $[\bs X, \bs o, \bs \mu]$ is unimodular, $k$ is an equivariant Markovian kernel and $\bs o'$ is a new root chosen with law $k_{\bs o}(\cdot)$. 
	If almost surely, $\bs \mu$ is a stationary measure for the Markovian kernel $k^{(\bs X, \bs \mu)}$  on $\bs X$, then $[\bs X, \bs o',\bs\mu]$ has the same distribution as $[\bs X, \bs o, \bs \mu]$.
\end{theorem}

In particular, in the compact case, this theorem implies the re-rooting invariance mentioned in Example~\ref{ex:finitemeasure}. This also generalizes the invariance of Palm distributions under bijective point-shifts and a similar statement for unimodular graphs (Proposition~3.6 of~\cite{eft}). See also Subsection~\ref{subsec:randomwalk} for a generalization to the case where an initial biasing is considered.


Before proving the theorem, we need some lemmas. We will first use the MTP when $k(\bs o)$ is absolutely continuous w.r.t. $\bs \mu$, and then, we will deduce the general case from the first case.

\begin{lemma}
	\label{lem:rerooting-cont}
	If $k_{\bs o}$ is absolutely continuous w.r.t. $\bs \mu$ almost surely, then the law of $[\bs X, \bs o',\bs\mu]$ is absolutely continuous w.r.t. the law of $[\bs X, \bs o,\bs\mu]$. In addition, if the Radon-Nikodym derivative $dk_{\bs o}/d\bs\mu$ is given by $f(\bs X, \bs o, \cdot, \bs\mu)$, where $f:\mdoublestar\to\mathbb R^{\geq 0}$ is a measurable function, then the distribution of $[\bs X, \bs o',\bs\mu]$ is obtained by biasing the distribution of $[\bs X, \bs o,\bs\mu]$ by $f^-(\bs o)$.
\end{lemma}

In fact, the proof shows that there always exists such a measurable Radon-Nikodym derivative $f$.

\begin{proof}
	Let $\alpha$ and $\beta$ be the $\sigma$-finite measures on $\mdoublestar$ defined by 
	\begin{eqnarray*}
		\alpha(A) &=& \omid{\int_{\bs X} \identity{A}[\bs X, \bs o, x,\bs\mu] d k_{\bs o} (x)},\\
		\beta(A) &=& \omid{\int_{\bs X} \identity{A}[\bs X, \bs o, x,\bs\mu] d\bs \mu(x)}.	
	\end{eqnarray*}
	$\alpha$ is just the distribution of $[\bs X, \bs o, \bs o',\bs\mu]$. It can be easily seen that $\alpha$ is absolutely continuous w.r.t. $\beta$. Let $f(X,o,o',\mu)$ be the Radon-Nikodym derivative of $\alpha$ w.r.t. $\beta$ at $[X,o,o',\mu]$. 
	It can be shown that $f$ satisfies the assumptions mentioned in the lemma.
	Now, let $A$ be an event in $\mstar$. One has
	\begin{eqnarray*}
		\myprob{[\bs X, \bs o',\bs \mu]\in A} &=& \omid{\int_{\bs X} \identity{A}[\bs X, x,\bs\mu] dk_{\bs o}(x)} \\
		&=& \omid{\int_{\bs X} f(\bs o, x)\identity{A}[\bs X, x,\bs\mu] d\bs \mu(x)}\\
		&=& \omid{\int_{\bs X} f(x, \bs o)\identity{A}[\bs X, \bs o,\bs\mu] d\bs \mu(x)}\\
		&=& \omid{\identity{A}[\bs X, \bs o,\bs \mu] f^-(\bs o)},
	\end{eqnarray*}
	where the third equality holds by the MTP~\eqref{eq:unimodular}. So the claim is proved.	
\end{proof}

\begin{lemma}
	\label{lem:balancingkernel}
	There exists a transport function $h$ such that $h$ is symmetric, $h>0$ and $h^+(\cdot)=h^-(\cdot)=1$ except on an event that has zero measure w.r.t. every nontrivial unimodular \rmm space $[\bs X, \bs o,\bs\mu]$. Indeed, if $g>0$ is an arbitrary transport function such that $g^+(\bs o)=1$ a.s., then one can let
	\begin{equation}
		\label{eq:balancingkernel}
		h(X,u,v,\mu):=\int_X \frac{g(u,x)g(v,x)}{g^-(x)} d\mu(x).
	\end{equation}
\end{lemma}
In fact, this is just the composition of the Markovian kernel corresponding to $g$ (given $(X,\mu)$) with its time-reversal. This ensures that the new kernel preserves $\mu$. 

\begin{proof}
	One can construct $g$ as follows. Given $(X,o,\mu)$, let $N>0$ be the smallest integer such that $\mu (\cball{N}{o})>0$ (unless $\mu=0$). For each $k\geq 1$, let $g_k(o,\cdot)$ be a constant function on $\cball{N+k}{o}$ such that $g_k^+(o)=1$. Then, let $g:=\sum_k 2^{-k} g_k$. Now, define $h$ by~\eqref{eq:balancingkernel}. Assuming 
	\begin{equation}
		\label{eq:lem:balancing}
		0<g^-(x)<\infty \text{ for } \bs\mu-\text{a.e. } x\in\bs X \text{ almost surely},
	\end{equation}
	it is straightforward to check that $h$ is well defined and $h^+(\bs o)=h^-(\bs o)=1$ a.s. So it remains to prove~\eqref{eq:lem:balancing}. Since $g>0$, one has $g^-(\bs o)>0$ a.s. Also, the MTP~\eqref{eq:unimodular} gives $\omid{g^-(\bs o)}=1$, and hence, $g^-(\bs o)<\infty$ a.s. So, Lemma~\ref{lem:happensatroot} implies \eqref{eq:lem:balancing} and the claim is proved.	
\end{proof}

\begin{remark}
	\label{rem:balacingkernel}
	Given any measurable function $b\geq 0$ on $\mstar$, one can modify the proof of Lemma~\ref{lem:balancingkernel} to have $h^+(\bs o)=h^-(\bs o)=b(\bs o)$ a.s. Note that in this case, the kernel has $b(\cdot)\bs{\mu}$ as a stationary measure (this is also readily implied by Lemma~\ref{lem:balancingkernel} and Example~\ref{ex:biasing}). Also, the kernel $\tilde h$ defined by $h$ (i.e., $\tilde{h}_{\bs o}:=b(\bs o)^{-1}h(\bs o, \cdot)\bs \mu$) can be arbitrarily close to the trivial kernel in the sense that $\prokhorov(\tilde{h}_{\bs o}, \delta_{\bs o})$ is less than an arbitrary constant $\epsilon>0$ a.s.	
\end{remark}

\begin{proof}[Proof of Theorem~\ref{thm:rerooting}]
	In the first case, assume that, as in Lemma~\ref{lem:rerooting-cont}, $k_{\bs o}(\cdot)$ is absolutely continuous w.r.t. $\bs \mu$ a.s. and its Radon-Nikodym derivative is given by $f:\mdoublestar\to\mathbb R^{\geq 0}$. Since $k_{\bs o}(\cdot)$ is a probability measure, $f^+(\bs o)=1$ a.s. Also, since the Markovian kernel $k^{(\bs X, \bs \mu)}$ preserves $\bs\mu$, one has $f^-(\cdot)=1$ a.e. on $\bs X$. Therefore, $f^-(\bs o)=1$ a.s. by Lemma~\ref{lem:happensatroot}. Therefore, by the second part of Lemma~\ref{lem:rerooting-cont}, $[\bs X, \bs o', \bs\mu]$ has the same distribution as $[\bs X, \bs o, \bs \mu]$.
	
Now consider the general case. By Lemma~\ref{lem:balancingkernel}, there exists a transport function $h>0$ such that $h^+(\bs o)=h^-(\bs o)=1$ a.s. Compose the Markovian kernel corresponding to $h$ with $k$ to obtain a new equivariant kernel $\eta$:
\[
	\eta_{\bs o}(\cdot):=\int_{\bs X} h(\bs o, x) k_x(\cdot)d\bs\mu(x).
\]
Now, $\eta$ is an equivariant Markovian kernel that preserves $\bs\mu$ a.s. and one can check that $\eta_{\bs o}(\cdot)$ is absolutely continuous w.r.t. $\bs \mu$ a.s. So, the first case implies that the composition of the two root changes preserves the distribution of $[\bs X, \bs o, \bs \mu]$. But the first root-change (corresponding to $h$) already preserves the distribution of $[\bs X, \bs o, \bs \mu]$ since $h^-(\bs o)=1$ a.s. This implies that the second root-change also preserves it and the claim is proved.
\end{proof}

\section{General Categories of Examples}
\label{sec:examples-general}

In this section, we present various general or specific examples in unimodular \rmm spaces. 
We also study the connection with Borel equivalence relations in Subsection~\ref{subsec:Borel}.

\subsection{Point Processes and Random Measures}
\label{subsec:pointprocess}

A \defstyle{point process} in $\mathbb R^d$ is a random discrete subset of $\mathbb R^d$. More generally, a \defstyle{random measure} on $\mathbb R^d$ is a random boundedly-finite measure on $\mathbb R^d$. A point process or random measure $\Phi$ is \defstyle{stationary} if $\Phi+t$ has the same distribution as $\Phi$ for every $t\in\mathbb R^d$. If so, the \defstyle{Palm version} of $\Phi$ can be defined in various ways and means heuristically \textit{seeing $\Phi$ from a typical point} or \textit{conditioning on $0\in\Phi$}. Formally, if $B\subseteq\mathbb R^d$ is an arbitrary bounded Borel set, bias by $\Phi(B)$ and move the origin to a random point with distribution $\restrict{\Phi}{B}$; i.e.,
\begin{equation}
	\label{eq:PalmEuclidean}
	\myprob{\Phi_0\in A} = \frac 1{\lambda_{\Phi}} \omid{\int_B \identity{A}(\Phi-x)d\Phi(x)},
\end{equation}
where $\lambda_{\Phi}$ is a constant called the \defstyle{intensity} of $\Phi$, which is assumed to be positive and finite here.
For other methods to defined the Palm version, one can mention the use of the \textit{Campbell measure} and \textit{tessellations}, which will be generalized in Section~\ref{sec:Palm}.  

The Palm version satisfies \textit{Mecke's theorem} \cite{Me67}, which can be rephrased equivalently\footnote{Mecke's theorem is stated in different notations in the literature, but it is easy to transform the equation in this form.} in the MTPs~\eqref{eq:MTP-pp} and~\eqref{eq:MTP-measure} (see also~\cite{shift-coupling} and~\cite{ThLa09}). Mecke proved in addition that with some finite moment condition, this equation characterizes the Palm versions of stationary point process. Also, one can reconstruct the stationary version from the Palm version via a formula which is known as \textit{Palm inversion}.
Without the finite moment condition, the MTPs \eqref{eq:MTP-pp} and~\eqref{eq:MTP-measure} define larger classes of point processes and random measures, which are called \defstyle{point-stationary point processes} and \defstyle{mass-stationary} random measures in~\cite{ThLa09}.\footnote{The definition of point-stationarity in~\cite{ThLa09} is more complicated, but it is proved in~\cite{ThLa09} that it is equivalent to Mecke's condition.}
For example, one can mention the zero set of the simple random walk, its graph, and the local time at zero of Brownian motion (see~\cite{I} for further examples). 
The MTPs~\eqref{eq:MTP-pp} and~\eqref{eq:MTP-measure} directly imply the following.

\begin{example}
	\label{ex:PPandRM}
	If $\Phi_0$ is the Palm version of a stationary point process in $\mathbb R^d$ (or more generally, a point-stationary point process), then $[\Phi_0,0]$ is a unimodular \rmm space (equipped with the counting measure). If $\Psi_0$ is the Palm version of a stationary random measure (or more generally, a mass-stationary random measure), then $[\mathrm{supp}(\Psi_0),0,\Psi_0]$ and $[\mathbb R^d,0,\Psi_0]$ are unimodular. Note that the latter may be improper.
\end{example}

\begin{remark}
	\label{rem:lostaxis}
	Since \rmm spaces are considered up to isomorphisms, some geometry is lost (e.g., the coordinate axes) when considering point processes and random measures as \rmm spaces. To fix this, one can extend $\mstar$ to \rmm spaces equipped with some additional geometric structure, which will be discussed in Subsection~\ref{subsec:extra}. In this sense, one can say that unimodular \rmm spaces generalize (point-) stationary point processes and (mass-) stationary random measures. 
	\\
	We will also provide a further generalization in Section~\ref{sec:Palm} by defining stationary point processes and random measures on a given unimodular \rmm space. In this viewpoint, unimodular \rmm spaces generalize the base space $\mathbb R^d$.
\end{remark}

\subsection{Unimodular Graphs and Discrete Spaces}

As mentioned in the introduction, a \defstyle{unimodular (random) graph}~\cite{processes} is a random rooted graph that satisfies the MTP~\eqref{eq:MTP-graph}. Also, the MTP provides many connections between unimodular graphs and (point-) stationary point process. As a common generalization, \cite{I} defines \defstyle{unimodular discrete spaces}, which are random boundedly-finite discrete metric spaces that satisfy a similar MTP. Since the Benjamini-Schramm topology for rooted graphs is consistent with the GHP topology (see e.g., \cite{Kh19generalization}), The MTP implies the following.

\begin{example}
	If $[\bs G, \bs o]$ is a unimodular graph or a unimodular discrete space, then it is also a unimodular \rmm space (equipped with the counting measure).
\end{example}

In considering graphs as \rmm spaces, some information might be lost (e.g., parallel edges and loops). However, one can keep these information by considering \rmm spaces equipped with some additional structure (see Subsection~\ref{subsec:extra}). In this sense, unimodular \rmm spaces generalize unimodular graphs.

More generally, \cite{processes} defines unimodular networks, which are unimodular graphs equipped with some marks on the vertices and edges. Also, \cite{I} defines unimodular marked discrete spaces. Similarly to the above example, these can be regarded as unimodular \rmm spaces equipped with suitable additional structures.

\begin{example}
	\label{ex:locallyunimodular}
	The following are examples of unimodular \rmm spaces which are graphs (or discrete spaces) equipped with a measure different from the counting measure. In these examples, the biasing is a special case of the Palm theory developed in Section~\ref{sec:Palm} (see Example~\ref{ex:biasing-Palm}).
	\begin{enumerate}[(i)]
		\item The Palm version of non-simple stationary point processes.
		\item If $[\bs G, \bs o]$ is a unimodular graph, then it is known that a stationary distribution of the simple random walk is obtained by biasing the probability measure by $\mathrm{deg}(\bs o)$. It can be seen that, after this biasing, $[\bs G, \bs o, \mathrm{deg}(\cdot)]$ is unimodular.
		\item Let $\bs S$ be the image of a process $(Y_n)_{n\in\mathbb Z}$ on $\mathbb R^d$ that has stationary increments and $Y_0=0$ (e.g., a random walk). Let $\bs \mu$ be the counting measure on $\bs S$ and $\bs m$ be the \textit{mutiplicity measure} on $\bs S$. Assuming that $\bs S$ is discrete and $\bs m$ is finite, $[\bs S, 0, \bs m]$ is unimodular (this is easily implied by the MTP on the index set $\mathbb Z$). However, to make $[\bs S, 0, \bs \mu]$ unimodular, one needs to bias by $\bs m(0)^{-1}$ (see Subsection~4.3 of~\cite{I}). 
		\item In some random rooted graphs $[\bs G, \bs o]$, the MTP holds only when the sum is made on a specific subset $\bs S$ containing the root. These cases are sometimes called \textit{locally unimodular} and are unimodular \rmm spaces by letting $\bs \mu$ be the counting measure on $\bs S$ (not on $\bs G$). An example is the graph of a null-recurrent Markov chain on $\mathbb Z$, where $\bs S$ is a level set of the graph. Another example appears in extending a unimodular graph as described in Example~\ref{ex:extension}. 
	\end{enumerate}
	
	
\end{example}

\subsection{Scaling Limits}
\label{subsec:scalinglimit}

Let $\bs G_n$ be a sequence of finite graphs (or metric spaces), which might be deterministic or random, and let $\bs o_n$ be a vertex in $\bs G_n$ chosen uniformly at random. A \defstyle{(measured) scaling limit} of $\bs G_n$ is the weak limit of a sequence of the form $[\epsilon_n \bs G_n, \bs o_n, \delta_n\bs{\mu}_n]$ as random elements in $\mstar$. Here, $\bs \mu_n$ is the uniform measure on the vertices of $\bs G_n$ and $\epsilon_n\bs G_n$ means that the graph-distance metric is scaled by $\epsilon_n$. Likewise, a \defstyle{subsequential (measured) scaling limit} is defined as the limit along a subsequence. The coefficients $\epsilon_n$ and $\delta_n$ may also depend on the non-rooted graph $\bs G_n$ (but not on $\bs o$). Since $\bs o_n$ is chosen uniformly, $[\bs G_n, \bs o_n]$ is a unimodular graph (in fact, it is enough to have the re-rooting invariance property, see Example~\ref{ex:finitemeasure}). 

Another setting for scaling limits is zooming-out a given unimodular graph or \rmm space $[\bs X, \bs o, \bs \mu]$; i.e., (subsequential) limits of $[\epsilon_n\bs X, \bs o, \delta_n \bs \mu]$ when the factors $\epsilon_n$ and $\delta_n$ do not depend on $\bs o$ and converge to zero appropriately. There are also examples of zooming-in a given \rmm space (even compact) at the root, which is the case when $\epsilon_n,\delta_n\to\infty$ appropriately. Various examples will be mentioned below.

In each of these settings, Lemma~\ref{lem:weaklimit} implies the following.

\begin{lemma}[Scaling Limit]
	\label{lem:scaling}
	Under the above assumptions, every measured scaling limit is unimodular and satisfies the MTP~\eqref{eq:unimodular}. In particular, if the limit is compact a.s., then it satisfies the re-rooting invariance property. The same is true for subsequential scaling limits, if the subsequence is chosen not depending on $\bs o$. 
\end{lemma}

Non-compact scaling limits have also some re-rooting invariance property, which is stated in Theorem~\ref{thm:rerooting}.

\begin{remark}
	Unimodularity does not make sense for non-measured scaling limits. However, in many cases, by a pre-compactness argument, it is possible to deduce the existence of a subsequential measured scaling limit. For instance, this is always the case if the scaling limit is compact.
\end{remark}

\begin{example}[Brownian Motion]
	The zero set of the simple random walk (SRW) on $\mathbb Z$ scales to the zero set of Brownian motion equipped with the local time measure. The graph of the SRW scales to the graph of Brownian motion with the measured induced from the time axis (the metric is distorted differently, but unimodularity is preserved anyway). The image of the SRW on $\mathbb Z^d$ ($d\geq 3$) scales to the image of Brownian motion on $\mathbb R^d$ equipped with the push-forward of the measure on $\mathbb R$. By Example~\ref{ex:locallyunimodular}, the latter is unimodular. These examples provide unimodular \rmm spaces (and also mass-stationary random measures).
\end{example}

\begin{example}[Brownian Trees]
	\label{ex:BCRT}
	The Brownian continuum random tree (BCRT) \cite{Al91-crtI} is the scaling limit of random trees on $n$ vertices. The re-rooting invariance property (observed in~\cite{Al91crtII}) is a direct corollary of Lemma~\ref{lem:scaling} above.
	Aldous also proved that by choosing larger scaling factors suitably, a non-compact scaling limit is obtained, which is called the \textit{self-similar continuum random tree} (SSCRT). Lemma~\ref{lem:scaling} implies that the SSCRT is also a unimodular \rmm space and satisfies the MTP, which seems to be a new result. 
\end{example}

\begin{example}[Stable Trees]
	Stable trees generalize the BCRT and are the scaling limit of Galton-Watson trees with infinite variance conditioned to be large~\cite{DuLe02}. A re-rooting invariance property of stable trees is proved in~\cite{DuLe05}. Note that the Galton-Watson trees are not re-rooting invariant. However, one can prove that the stable trees are the scaling limits of critical \textit{unimodular Galton-Watson (UGW) trees} (Example~1.1 of~\cite{processes}) as well. Since UGW trees are unimodular, the re-rooting invariance property is implied by Lemma~\ref{lem:weaklimit} directly.
%
%
\end{example}

\begin{example}[Self-Similar Unimodular Spaces]
	\label{ex:selfsimilar}
	Let $K\subseteq\mathbb R^d$ be a self-similar set such that $K=\cup f_i(K)$, where each $f_i$ is a homothety (see~\cite{bookBiPe17}). Equip $K$ with the unique self-similar probability measure $\mu$ on it and choose $\bs o\in K$ with distribution $\mu$. Assume all homothety ratios are equal and the \textit{open set condition} holds. In this case, by zooming-in at $\bs o$, there exists a subsequential scaling limit. This can be proved similarly to Theorem~4.14 of~\cite{I} (regarding \textit{self-similar unimodular discrete spaces}) and the proof is skipped for brevity (the limit can be constructed by adding to $K$ some isometric copies and continuing recursively, similarly to Remark~4.21 of~\cite{I}). In addition, the limit is the same as the subsequential scaling limit when zooming-out the unimodular discrete self-similar set defined in~\cite{I}. Hence, the scaling limit is unimodular. 
\end{example}

\begin{example}[Micro-Sets]
	More general than Example~\ref{ex:selfsimilar}, \defstyle{micro-sets} are defined for an arbitrary compact set $K\subseteq\mathbb R^d$ by zooming-in at a given point (see Subsection~2.4 of~\cite{bookBiPe17}). Let $\mu$ be a probability measure on $K$ and choose $\bs o$ randomly with distribution $\mu$. By modifying the definition of micro-sets slightly to allow non-compact micro-sets (using convergence of closed subsets of $\mathbb R^d$), and also by scaling $\mu$ at the same time, one obtains that every subsequential measured micro-set at $\bs o$ defined as a weak limit (the subsequence and the scaling parameters should not depend on $\bs o$) is unimodular.
\end{example}

\begin{example}[Brownian Web]
	Brownian web is the scaling limit of various \textit{drainage network models} (see e.g., \cite{FoNeRa04}), which are random trees embedded in the plane. The limit is usually defined with a different notion of convergence (as collections of paths in the plane), but it is also proved recently in~\cite{CaHa21brownianweb} that the scaling limit exists in the Gromov-Hausdorff topology as well. It is natural to expect that the measured scaling limit also exists. Nevertheless, the limit is the completion of the \textit{skeleton} of the Brownian web, and hence, has a natural measure induced from the Lebesgue measure on $\mathbb R^2$. By stationarity, the resulting measured continuum tree is a unimodular \rmm space.
\end{example}

\begin{example}[Uniform Spanning Forest]
	Let $\bs C$ be the connected component containing 0 of the \textit{uniform spanning forest} of $\mathbb Z^d$ (see Section~7 of~\cite{processes}). It is known that $[\bs C, 0]$ is unimodular (the same is true in any unimodular graph); indeed, it is point-stationary. 
	As random closed subsets of $\mathbb R^d$ equipped with a measure, one can use precompactness to show that there exists subsequential measured scaling limits. If one proves the existence of the scaling limit (or if the subsequence is chosen in a translation-invariant way), then the limit is a unimodular \rmm space. 
	It is also conjectured that the scaling limit of $\bs C$ exists as random real trees embedded in $\mathbb R^d$, and for $d\geq 5$, the limit is identical to the Brownian-embedded SSCRT (see Example~\ref{ex:BCRT} and~\cite{vanderHo06}).
\end{example}

\begin{example}[Planar Maps]
	The scaling limit of random planar graphs and maps have been of great interest in probability theory and in physics. For instance, the \textit{Brownian map} and the \textit{Brownian disk} are compact random metric spaces arising as the scaling limits of some models of uniform random planar triangulations and quadrangulations (see e.g., \cite{Le19}). The \textit{Brownian plane} is a non-compact model obtained by zooming-out the \textit{uniform infinite planar quadrangulation} and also by zooming-in the Brownian map. Therefore, the Brownian plane, equipped with the \textit{volume measure} (i.e., the scaling limit of the counting measure), is a unimodular \rmm space. 
	\\
	In addition, \cite{Bu18} defines the \textit{hyperbolic Brownian plane} as a limit of a sequence of planar triangulations scaled properly. It is observed in~\cite{Bu18} that this model satisfies a version of the MTP. 
	Indeed, this means that the hyperbolic Brownian plane is a unimodular \rmm space and is a direct corollary of Lemma~\ref{lem:weaklimit}.
	\\
	Also, in the uniform infinite \textit{half-plane} triangulation, the root is on the \textit{boundary}, which is a bi-infinite path. It is proved that the scaling limit of this model exists. Equipping it with the \textit{length measure}, which is the scaling limit of the counting measure on the boundary, a unimodular \rmm space is obtained (which is improper). 
\end{example}

\subsection{Groups and Deterministic Spaces}

In this subsection, we investigate when a deterministic \rmm space $[X,o,\mu]$ is a unimodular \rmm space. By Lemma~\ref{lem:happensatroot}, it is necessary that $(X,\mu)$ is \textit{transitive}; i.e., $[X,o,\mu]$ is isomorphic to $[X,y,\mu]$ for every $y\in X$. However, transitivity is not enough. The following generalizes the analogous result about transitive graphs (see~\cite{BLPS}).

\begin{proposition}
	\label{prop:fixed1}
	Let $[X,o,\mu]$ be a deterministic \rmm space.
	\begin{enumerate}[(i)]
		\item \label{thm:fixed1-i} $[X,o,\mu]$ is unimodular if and only if $(X,\mu)$ is transitive and the automorphism group of $(X,\mu)$ is a unimodular group.
		\item \label{thm:fixed1-ii} Assume $\Gamma$ is a closed subgroup of the automorphism group of $(X,\mu)$. Then, the MTP holds for all $\Gamma$-invariant functions on $X\times X$ if and only if $\Gamma$ is unimodular and acts transitively on $X$.
	\end{enumerate}
\end{proposition}
The first part is proved in a more general form in Theorem~ \ref{thm:fixed2}. The second part can also be proved similarly and the proof is skipped for brevity.


\begin{corollary}[Unimodular Groups]
	Every unimodular group $\Gamma$ equipped with the Haar measure and a boundedly-compact left-invariant metric $d$ (e.g., any finitely generated group equipped with a Cayley graph) is a unimodular \rmm space.
\end{corollary}
It is also easy to prove this corollary by verifying the MTP directly.

\begin{example}
	The Euclidean space $\mathbb R^n$ is a unimodular group, and hence, $[\mathbb R^d,0,\mathrm{Leb}]$ is unimodular. The hyperbolic space $\mathbb H^n$ is not a group, but Proposition~\ref{prop:fixed1} implies that it is unimodular. However, the hyperbolic plane with one distinguished ideal point (which can be regarded as a \rmm space with some additional structure; see Subsection~\ref{subsec:extra}) is not unimodular since its automorphism group is not unimodular.
\end{example}


\subsection{Quasi-Transitive Spaces}

Let $(X,\mu)$ be a deterministic measured metric space. In this subsection, we investigate when one can choose a random root $\bs o\in X$ such that $[X, \bs o, \mu]$ is unimodular. This generalizes the case of quasi-transitive graphs;\footnote{This is the only motivation for the title of this subsection.} see~\cite{BLPS} and Section~3 of~\cite{processes}. 

Let $\Gamma$ be the automorphism group of $(X,\mu)$. If $\Gamma$ contains only the identity function, then it is straightforward to see that $\mu$ should be a finite measure and the distribution of $\bs o$ should be proportional to $\mu$ (see Example~\ref{ex:finitemeasure}). However, if $\Gamma$ is nontrivial, the situation is more involved. 

The \defstyle{orbit} of $x\in X$ is the set $\Gamma x = \{\gamma x: \gamma \in \Gamma\}$. 
%
%
Boundedly-compactness of $X$ implies that $\Gamma$ is a locally-compact topological group. Let $\norm{\cdot}$ be a left-invariant Haar measure on $\Gamma$.
Let $B$ be an arbitrary open set that intersects every orbit in a nonempty bounded set. For instance, one can let $B:=\bigcup_{n\geq 1} \oball{n}{o}\setminus \Gamma \cball{n-2}{o}$ for an arbitrary $o\in X$. For $x\in X$, define 
\[
	h(x):=\norm{\Gamma_{x,B}}^{-1}, \text{ where } \Gamma_{x,B}:=\{\gamma\in \Gamma: \gamma^{-1} x\in B\}.
\]
The assumptions imply $0<h(x)<\infty$. Also, for $\alpha\in\Gamma$, one has $\Gamma_{\alpha x,B} = \alpha \Gamma_{x,B}$, and hence, $h(\alpha x) = h(x)$. 
\begin{theorem}
	\label{thm:fixed2}
	There exists a random point $\bs o\in X$ such that $[X,\bs o,\mu]$ is unimodular if and only if $\Gamma$ is a unimodular group and $\int_B hd\mu<\infty$. 
	In this case, the distribution of $[X, \bs o, \mu]$ is unique. In addition, $\bs o$ can be chosen with distribution proportional to $\restrict{(h\mu)}{B}$.
\end{theorem}
 
This generalizes Theorem~3.1 of~\cite{processes}. Indeed, in the case where $X$ is a graph or network and $\mu$ is the counting measure, one can choose $B$ by choosing exactly one point from every orbit. In this case, for $x\in B$, $h(x)$ is the inverse of the measure of the \textit{stabilizer} of $x$. 
One can also generalize the \textit{if} side of the claim to the case where $\Gamma$ is a closed subgroup of the automorphism group that acts transitively.


\begin{example}
	Let $X$ be a \textit{horoball} in the hyperbolic plane corresponding to an ideal point $\omega$ and let $\mu$ be the volume measure on $X$. Theorem~\ref{thm:fixed2} shows that one can choose a random point $\bs o\in X$ such that $[X, \bs o, \mu]$ is unimodular. Indeed, if $B$ is the region between any two lines passing through $\omega$, then it is enough to choose $\bs o$ uniformly in $B$ (note that $\Gamma$ is isomorphic to the automorphism group of $\mathbb R$). This can be thought of as a continuum version of the \textit{canopy tree}.
	\\
	For a simpler example, let $\alpha$ be a finite measure on $\mathbb R$ and let $\mu:=\alpha\times \mathrm{Leb}$. Then, $[\mathbb R^2, \bs o, \mu]$ is unimodular, where $\bs o$ is chosen on the $x$ axis with distribution proportional to $\alpha$.
\end{example}

\begin{proof}[Proof of Theorem~\ref{thm:fixed2}]
	For every measurable function $f:X\to\mathbb R^{\geq 0}$, one has
	\begin{equation*}
		\label{eq:thm:fixed2}
		\int_X f(y)d\mu(y) = \int_B \int_{\Gamma}h(x) f(\gamma x) d\gamma d\mu(x).
	\end{equation*}
	This can be seen by the change of variable $y:=\gamma x$ in the right hand side and noting that $h$ and $\mu$ are $\Gamma$-invariant. First, assume that $a:=\int_B h(x)d\mu(x)<\infty$ and $\bs o\sim\frac 1 a\restrict{(h\mu)}{B}$.  Therefore, for every transport function $g$,
	\begin{eqnarray*}
		\nonumber a \omid{\int g(\bs o, y) d\mu(y)} &=& \int_B h(z)\int_X g(z,y)d\mu(y)d\mu(z)\\
		\nonumber &=& \int_B h(z) \int_B \int_{\Gamma} h(x)  g(z,\gamma x) d\gamma d\mu(x)  \\
		\label{eq:thm:fixed2-1} &=& \int_{\Gamma}\int_B\int_B h(z)h(x) g(z,\gamma x) d\mu(x) d\mu(z) d\gamma.
	\end{eqnarray*}
	Similarly, 
	\begin{eqnarray*}
		\label{eq:thm:fixed2-2} a \omid{\int g(y, \bs o) d\mu(y)} &=& \int_{\Gamma}\int_B\int_B h(z)h(x) g(\gamma z,x) d\mu(x) d\mu(z) d\gamma.
	\end{eqnarray*}
	Note that $g(\gamma z, x) = g(z, \gamma^{-1} x)$. If $\Gamma$ is unimodular, then the change of variable $\gamma\to \gamma^{-1}$ preserves the Haar measure. This implies that the two above formulas are equal and the unimodularity of $[X, \bs o, \mu]$ is proved.
	
	Conversely, assume $\bs o\in X$ is a random point such that $[X, \bs o,\mu]$ is unimodular. 
	Let $\mathbb P$ and $\nu$ be the distributions of $\bs o$ and $[X,\bs o,\mu]$ respectively. For $A\subseteq X$, define $[A]:=\{[X,u,\mu]:u\in A\}$. So $\nu([A]) = \myprob{\bs o\in \Gamma A}$.
	Let $C,D\subseteq X$ be arbitrary and define $g(u,v):=\norm{\{\gamma: \gamma^{-1} u\in C, \gamma^{-1} v\in D\}}$. One has
	\begin{eqnarray*}
		\omid{\int_X g(\bs o,y)d\mu(y)} &=& \omid{\int_X\int_{\Gamma}\identity{C}(\gamma^{-1}\bs o)\identity{D}(\gamma^{-1} y)d\gamma d\mu(y)}\\
		&=& \omid{\int_X\int_{\Gamma}\identity{C}(\gamma^{-1}\bs o)\identity{D}(y)d\gamma d\mu(y)}\\
		&=& \mu(D)\ \omid{\int_{\Gamma}\identity{C}(\gamma^{-1}\bs o)d\gamma},
	\end{eqnarray*}
	where the second equality is by the change of variable $y':=\gamma^{-1}y$. Similarly,
	\begin{eqnarray*}
		\omid{\int_X g(y,\bs o)d\mu(y)} &=& \mu(C)\ \omid{\int_{\Gamma}\identity{D}(\gamma^{-1}\bs o)d\gamma}.
	\end{eqnarray*}	
	Since $g$ is $\Gamma$-invariant, the MTP implies that these two equations are equal. Since it holds for arbitrary $C$ and $D$, there exists a constant $b$ such that  
	\begin{equation}
		\label{eq:thm:fixed2-3}
		\omid{\int_{\Gamma}\identity{C}(\gamma^{-1}\bs o)d\gamma} = b \mu(C), \quad \forall C\subseteq X.
	\end{equation}
	Thus, for every measurable function $f$ on $X$, one has $\omid{\int_{\Gamma}f(\gamma^{-1}\bs o)d\gamma} = b \int_X f d\mu$.
	In particular, let $f=h\identity{A} \identity{B}$, where $A\subseteq X$ is any $\Gamma$-invariant set. Hence,
	\begin{eqnarray*}
		b \int_X h \identity{A} \identity{B} d\mu &=& \omid{\int_{\Gamma}h(\gamma^{-1}\bs o)\identity{A}(\gamma^{-1}\bs o) \identity{B}(\gamma^{-1}\bs o)d\gamma}\\
		&=& \omid{h(\bs o)\identity{A}(\bs o) \int_{\Gamma} \identity{B}(\gamma^{-1}\bs o)d\gamma}\\
		&=& \omid{\identity{A}(\bs o)} = \nu([A]).
	\end{eqnarray*}
	This proves the uniqueness of $\nu$. Also, by letting $A:=X$, one gets that $\int_B h d\mu<\infty$ and $b=1/a$. It also implies that, by choosing another point with distribution $b\restrict{(h\mu)}{B}$, $\nu$ would not change. Hence, one may choose $\bs o\sim b\restrict{(h\mu)}{B}$ from the beginning. In addition, for every $\beta\in\Gamma$, since $\mu(C) = \mu(\beta C)$, \eqref{eq:thm:fixed2-3} implies
	\begin{eqnarray*}
		\omid{\int_{\Gamma}\identity{C}(\gamma^{-1}\bs o)d\gamma} =  \omid{\int_{\Gamma}\identity{\beta C}(\gamma^{-1}\bs o)d\gamma} = m(\beta)^{-1} \omid{\int_{\Gamma}\identity{C}(\gamma^{-1}\bs o)d\gamma},
	\end{eqnarray*}	
	where the last equality is by the change of variable $\gamma':=\gamma\beta$ and this changes the Haar measure by the constant factor $m(\beta)$, where $m$ is the \textit{modular function} of $\Gamma$. The above equation implies that $m(\beta)=1$; i.e., $\Gamma$ is unimodular. So the claim is proved.
\end{proof}

\subsection{Connection with Borel Equivalence Relations}
\label{subsec:Borel}

Let $S$ be a Polish space and $E$ be an equivalence relation on $S$. It is called a \defstyle{countable Borel equivalence relation} (CBER) if it is a Borel subset of $S\times S$ and every equivalence class is countable. A probability measure $\nu$ on $S$ is \defstyle{invariant} under $R$ if for all measurable functions $f:S\times S\to\mathbb R^{\geq 0}$, one has $\int \sum_{y\in R(x)} f(x,y) d\nu(x) = \int \sum_{y\in R(x)} f(y,x)d\nu(x)$, where $R(x)$ is the equivalence class containing $x$. As discussed in Example~9.9 of~\cite{processes}, this notion is tightly connected to unimodular graphs: 
if one has a \textit{graphing} of $R$ and $\bs o\in X$ is a random point with distribution $\nu$, then the component of the graphing containing $\bs o$ is unimodular. Conversely, if $\mathbb P$ is a unimodular probability measure on $\mathcal G_*$, where $\mathcal G_*\subseteq\mstar$ is the space of connected rooted graphs, then $\mathbb P$ is invariant under the equivalence relation on $\mathcal G_*$ defined by $(G,o)\sim (G,v)$ for all $v\in G$. The last claim is an \textit{if and only if} in the case where $\mathbb P$ is supported on graphs with no nontrivial automorphism. As mentioned in~\cite{processes}, there is a substantial overlap between the two theories, but their viewpoints and motivations are different. In fact, graphings of CBERs are quite more general and they can capture some features of graph limits that unimodular graphs don't (see \textit{local-global convergence} in~\cite{HaLoSz14}). Also, some of the results for unimodular graphs are proved by the results on CBERs; e.g., ergodic decomposition and amenability. It should be noted that, due to the possibility of automorphisms, some geometry is lost when passing to graphings. This issue should be dealt with; e.g., by adding extra randomness to break the automorphisms. 

For unimodular \rmm spaces, the analogous equivalence class on $\mstar$ defined by $(X,o,\mu)\sim (X,v,\mu)$ for $v\in X$ is not countable. Hence, the theory of CBERs is not directly applicable. In Section~\ref{sec:Palm-application}, we will construct a CBER by introducing the \textit{Poisson point process} on unimodular \rmm spaces. Also, by the Palm theory developed in Section~\ref{sec:Palm}, we construct an invariant measure for the CBER. This enables us to use the results in the theory of CBERs. 

\section{Further Properties of Unimodular \rmm Spaces}
\label{sec:further}

\subsection{Allowing Extra Randomness}
\label{subsec:extra}

Lemma~\ref{lem:happensatroot} deals with factor subsets, which are functions of the underlying measured metric space. The lemma still holds if one allows extra randomness in choosing the subset suitably if unimodularity is preserved. However, this is not straightforward to formalize due to measurabiliy issues (the space of all $(X,o,\mu,S)$, where $S$ is a Borel subset of $X$, does not have a natural topology).\footnote{This is similar to the issue of defining random Borel subsets of $\mathbb R^d$, where one has to use a random field instead. The following idea is similar to the use of random fields.} 

To do this generalization, assume there exists already a random geometric structure $\bs m$ on $\bs X$ and $[\bs X, \bs o, \bs \mu, \bs m]$ is a random \rmm space equipped with some additional structure. This makes sense as soon as there is a suitable generalization of the GHP metric. For instance, $\bs m$ can be a random measure on $\bs X$, a random closed subset of $\bs X$, etc. In the previous work~\cite{Kh19generalization}, a quite general framework is presented for extending the GHP metric using the notion of functors. This can be applied to various types of additional geometric structures and sufficient criteria for Polishness are also provided (it is necessary that for every deterministic $(X,o,\mu)$, the set of additional structures on $(X,o,\mu)$ is Polish, but more assumptions are needed). Here, we assume that Polishness holds as well.

\begin{definition}
	$[\bs X, \bs o, \bs \mu, \bs m]$ is unimodular if the MTP~\eqref{eq:unimodular} holds even if $g$ depends on the additional structure $\bs m$.
\end{definition}

In fact, this means that there exists a version of the conditional distribution of $\bs m$ given $[\bs X, \bs o, \bs \mu]$ such that the conditional law does not depend on the root. This will be formalized in Section~\ref{sec:Palm} (see Definition~\ref{def:equivmeasure}, Lemma~\ref{lem:equivmeasure} and Remark~\ref{rem:disintegration}).

Once we have such an additional structure, then one can extend the notion of factor subsets by allowing it to be a function of $(\bs X, \bs \mu, \bs m)$. Then, the results of this paper like Lemmas~\ref{lem:happensatroot} and~\ref{lem:weaklimit} can be generalized to this more general setting.

Additional geometric structures are interesting in their own and various special cases have been considered in the literature separately (see~\cite{Kh19generalization} for an account of the literature and for a unification). In this work, we use this setting in various places; e.g., for developing Palm theory in Section~\ref{sec:Palm} (where $\bs m$ is a tuple of $k$ random measures on $\bs X$), for studying random walks on unimodular spaces in Subsection~\ref{subsec:randomwalk} (where $\bs m$ is a sequence of points of $\bs X$), for ergodicity (Subsection~\ref{subsec:ergodic}), for hyperfiniteness (Subsection~\ref{subsec:amenable}), and in the proofs in Section~\ref{sec:Palm-application} (where $\bs m$ is a marked measure or a closed subset).

\subsection{Random Walk}
\label{subsec:randomwalk}

In Subsection~\ref{subsec:rerooting}, re-rooting a unimodular \rmm space is defined by means of equivariant Markovian kernels. By iterating a re-rooting, one obtains a random walk on the unimodular \rmm space, as formalized below. 

Let $\mathcal M_{\infty}$ be the space of all $(X,o,\mu, (y_n)_{n\in\mathbb Z})$, where $(y_n)_n$ is a sequence in $X$ and $y_0=o$. By the discussion in Subsection~\ref{subsec:extra}, one can show that $\mathcal M_{\infty}$ can be turned into a Polish space. Consider the following shift operator on $\mathcal M_{\infty}$:
\begin{equation}
	\label{eq:shift}
	\mathcal S(X,o,\mu,(y_n)_n):=(X,y_1,\mu,(y_{n+1})_n).
\end{equation}
	
An \defstyle{initial bias} is a measurable function $b:\mstar\to\mathbb R^{\geq 0}$. Then, for every deterministic $(X,\mu)$, one obtains a measure $b\mu$ on $X$ defined by
\begin{equation}
	\label{eq:b mu}
	b\mu(A):=\int_A b(X,y,\mu)d\mu(y).
\end{equation}

Let $k$ be an equivariant Markovian kernel (Subsection~\ref{subsec:rerooting}). Given every deterministic $(X,o,\mu)$, the kernel $k^{(X,\mu)}$ defines a Markov chain $(\bs x_n)_{n\in\mathbb Z}$ on $X$ such that $\bs x_0=o$ and $\bs x_1$ has law $k_o$. 
Let $\theta_{(X,o,\mu)}$ be the law of $(X,o,\mu,(\bs x_n)_n)$ on $\mathcal M_{\infty}$. 
Let $[\bs X, \bs o, \bs \mu]$ be a random \rmm space such that $\omid{b(\bs o)}<\infty$ and let $Q$ be the distribution of $[\bs X, \bs o, \bs \mu, (\bs x_n)_n]$ biased by $b(\bs o)$. 

\begin{theorem}
	\label{thm:randomwalk}
	Assume $[\bs X, \bs o, \bs \mu]$ is a unimodular \rmm space. 
	Under the above assumptions, if for almost every sample $[X,o,\mu]$ of $[\bs X, \bs o, \bs \mu]$, $b{\mu}$ is a stationary (resp. reversible) measure for the Markovian kernel $k^{(X, \mu)}$ on $X$, then $Q$ is a stationary (resp. reversible) measure under the shift $\mathcal S$. 
\end{theorem}
This generalizes Theorem~4.1 of~\cite{processes} which is for unimodular graphs. A further generalization is provided in Subsection~\ref{subsec:Palm-randomwalk} by allowing the stationary measure be singular with respect to $\bs \mu$.

\begin{proof}
	First, assume that $k_{\bs o}$ is absolutely continuous w.r.t. $\bs \mu$ almost surely. In this case, a slight modification of the proof of Theorem~4.1 of~\cite{processes} (similarly to Lemma~\ref{lem:rerooting-cont}) can be used to prove the claim. But in the general case, the idea of the proof of Theorem~\ref{thm:rerooting} (composing $k$ with a continuous kernel) does not work. 
	In this case, it is enough to approximate $k$ by a sequence of equivariant kernels that have the same properties as $k$ (stationarity or reversibility) plus absolute continuity. For this, let $\tilde h_n$ be a sequence of kernels converging to the trivial kernel given by Remark~\ref{rem:balacingkernel}. Note that the kernel $\tilde h_n \circ k \circ \tilde h_n$ converges to $k$ as $n\to\infty$, preserves $b\bs\mu$ if $k$ does, is reversible if $k$ is reversible (since $h_n$ is symmetric) and has the absolute continuity property. This proves the claim.
%
%
%
\end{proof}

\begin{theorem}[Characterization of Unimodularity]
	\label{thm:randomwalk-characterization}
	Let $h$ be a fixed symmetric transport function as in Lemma~\ref{lem:balancingkernel} ($h>0$ and $h^+(\cdot)=h^-(\cdot)=1$). Let $(\bs x_n)_n$ be the random walk given by the kernel $k$ defined by $k_o:=h(o,\cdot)\mu$. Then, a random \rmm space $[\bs X, \bs o, \bs \mu]$ is unimodular if and only if the law of $[\bs X, \bs o, \bs \mu, (\bs x_n)_n]$ (defined above) is stationary and reversible under the shift $\mathcal S$. The condition $h>0$ can also be relaxed to $\sum_n h^{(n)}>0$, where $h^{(n)}$ is given by the $n$-fold composition of the kernel with itself.
\end{theorem}
\begin{remark}
	One can similarly extend this theorem to have an arbitrary initial bias $b$ by having $h^+=h^-=b$. 
	This generalizes the fact that for random rooted graphs, unimodularity is equivalent to stationarity and reversibility of the simple random walk after biasing by the degree of the root (see Section~4 of~\cite{processes}). The benefit of the above theorem is that it characterizes all unimodular \rmm spaces, not those with the moment condition $\omid{b(\bs o)}<\infty$. This seems to be new even for graphs (see also the proof of Lemma~4 of~\cite{shift-coupling}).
\end{remark}

\begin{proof}[Proof of Theorem~\ref{thm:randomwalk-characterization}]
	The only if part is implied by Theorem~\ref{thm:randomwalk}. For the if part, note that if $h>0$, then the two sides of~\eqref{eq:unimodular} are $\omid{g'(\bs o, \bs x_1)}$ and $\omid{g'(\bs x_1, \bs o)}$, where $g':=g/h$. So, since $[\bs X, \bs o, \bs x_1,\bs\mu]$ has the same distribution as $[\bs X, \bs x_1,\bs o, \bs \mu]$ by reversibility, the MTP holds; i.e., $[\bs X, \bs o, \bs \mu]$ is unimodular.
	
	Under the weaker condition $\sum_n h^{(n)}>0$, let $A_n$ be the event $h^{(n)}>0$ and write $g=\sum_n g_n$, where $g_n$ is the restriction of $g$ to $A_n\setminus (\cup_{i=1}^{n-1} A_i)$. Each $g^{(n)}$ satisfies the MTP by the first part of the proof. Hence, $g$ also satisfies the MTP.
\end{proof}

\begin{proposition}[Speed Exists]
	Under the setting of Theorem~\ref{thm:randomwalk}, the \textit{speed} of the random walk $(\bs x_n)_n$, defined by $\lim_n d(\bs o, \bs x_n)/n$, exists.
\end{proposition}
This is a generalization of Proposition~4.8 of~\cite{processes} and can be proved by the same argument using Kingman's subadditive ergodic theorem. In fact, Theorem~\ref{thm:ergodicity-events} implies that the speed does not depend on $\bs o$ and is measurable with respect to the invariant sigma-field (up to modifying a null event).

\subsection{Ergodicity}
\label{subsec:ergodic}

An even $A\subseteq\mstar$ is \defstyle{invariant} if it does not depend on the root; i.e., if $[X,o,\mu]\in A$, then $\forall y\in X: [X,y,\mu]\in A$. Let $I$ be the sigma-field of invariant events. A unimodular \rmm space $[\bs X, \bs o, \bs \mu]$ (or a unimodular probability measure on $\mstar$) is  \defstyle{ergodic} if $\myprob{A}\in\{0,1\}$ for all $A\in I$.

One can express ergodicity in terms of ergodicity of random walks as follows. Let $h$ be a fixed symmetric transport function such that $h^+(\cdot)=h^-(\cdot)=1$ and $h>0$ (given by Lemma~\ref{lem:balancingkernel}). Let $(\bs x_n)_{n\in\mathbb Z}$ be the resulting two-sided random walk as in Theorem~\ref{thm:randomwalk-characterization} and let $Q$ be the distribution of $[\bs X, \bs o, \bs \mu, (\bs x_n)_n]$. Let $\Gamma$ be the group of automorphisms of the time axis $\mathbb Z$; i.e., transforms of the form $t\mapsto t+t_0$ or $t\mapsto-t+t_0$. By Theorem~\ref{thm:randomwalk-characterization}, $[\bs X, \bs o, \bs \mu]$ is unimodular if and only if $Q$ is invariant under the action of $\Gamma$ on $\mathcal M_{\infty}$ defined by shifts and time-reversion (note that considering time-reversions is important at this point). Here, we prove:

\begin{theorem}
	\label{thm:ergodicity-equivalence}
	$[\bs X, \bs o, \bs \mu]$ is ergodic if and only if the law of $[\bs X, \bs o, \bs \mu, (\bs x_n)_n]$ is ergodic under the action of $\Gamma$.
\end{theorem}

The latter means that every $\Gamma$-invariant event has probability zero or one. This result is straightforwardly implied by the following theorem.

\begin{theorem}
	\label{thm:ergodicity-events}
	Under the above setting, for every shift-invariant event $B\subseteq\mathcal M_{\infty}$, there exists an invariant event $A\in I$ such that $B\Delta \{(X,o,\mu,(x_n)_n): (X,o,\mu)\in A\}$ has zero probability for every unimodular \rmm space equipped with the random walk $(\bs x_n)_n$.
\end{theorem}

This extends Theorem~4.6 of~\cite{processes} with the additional point that $A$ does not depend on the choice of the unimodular probability measure (this is important in the next section). The same proof works here and is skipped for brevity. This results hold for an arbitrary random walk preserving an initial bias as well (as in Theorem~\ref{thm:randomwalk}). Also, one can relax the condition $h>0$ to $\sum_n h^{(n)}>0$ as in Theorem~\ref{thm:randomwalk-characterization}.

\begin{corollary}
	\label{cor:average}
	Let $[\bs X, \bs o, \bs \mu]$ and $(\bs x_n)_n$ be as above. Then, for every $f:\mstar\to\mathbb R$ such that $\omid{\norm{f(\bs o)}}<\infty$, $\mathrm{ave}(f):=\lim_n \frac 1 {2n}\sum_{i=-n}^n f(\bs x_n)$ exists and $\omid{\mathrm{ave}(f)}=\omid{f(\bs o)}$. In particular, if $[\bs X, \bs o, \bs \mu]$ is ergodic, then $\mathrm{ave}(f)=\omid{f(\bs o)}$ a.s.
\end{corollary}

\subsection{Ergodic Decomposition}
\label{subsec:ergodicDecomposition}

In this subsection, we will prove \textit{ergodic decomposition} for unimodular \rmm spaces; i.e., expressing the distribution as a mixture of ergodic probability measures. This does not follow immediately from the existing results in the literature; e.g., for measure preserving semi-group actions (see~\cite{Fa62}) or for countable Borel equivalence relations. 
We will deduce the claim from ergodic decomposition for the random walk.

Let $\mathcal U$ (resp. $\mathcal E$) denote the set of unimodular (resp. ergodic) probability measures on $\mathcal M_*$. These are subsets of the set of probability measures on $\mstar$ and can be equipped with the topology of weak convergence. Also, $\mathcal U$ is a closed and convex subset by Lemma~\ref{lem:weaklimit}. 

\begin{proposition}
	\label{prop:extreme}
	$\mathcal E$ is the set of \textit{extreme points} of $\mathcal U$.
\end{proposition}
An \textit{extreme point} means a point that cannot be expresses as a a convex combination of two other points. This claim is implied by the following stronger result and the proof is skipped.

\begin{theorem}[Ergodic Decomposition]
	\label{thm:ergodicDecomposition}
	For every unimodular probability measure $\nu$ on $\mstar$, there exists a unique probability measure $\lambda$ on $\mathcal E$ such that $\nu = \int_{\mathcal E} \varphi d\lambda(\varphi)$.
\end{theorem}

See also Theorem~\ref{thm:ergodicDecomposition2} for another formulation.
The proof is by using random walks. 
Let $h$ be an arbitrary symmetric transport function as in Lemma~\ref{lem:balancingkernel} ($h>0$ and $h^+(\cdot)=h^-(\cdot)=1$) and let $(\bs x_n)_n$ be the resulting random walk as in Subsection~\ref{subsec:randomwalk}. Let $I'$ be the sigma field of $\Gamma$-invariant events in $\mathcal M_{\infty}$, where $\Gamma$ is the automorphism group of $\mathbb Z$ (as in Subsection~\ref{subsec:ergodic}). Let $\mathcal U'$ (resp. $\mathcal E'$) be the set of $\Gamma$-invariant probability measures on $\mathcal M_{\infty}$ (resp. ergodic under the action of $\Gamma$). Let $\pi:\mathcal M_{\infty}\to\mstar$ be the projection defined by forgetting the trajectory of points. By using Theorems~\ref{thm:randomwalk-characterization} and~\ref{thm:ergodicity-equivalence}, it is straightforward to deduce that $\pi_* \mathcal U' = \mathcal U$ and $\pi_*\mathcal E'=\mathcal E$.

\begin{proof}[Proof of Theorem~\ref{thm:ergodicDecomposition} (Existence)]
	Let $\mu$ be the distribution of a unimodular \rmm space $[\bs X, \bs o, \bs \mu]$. Let $\hat{\nu}$ be the distribution of $[\bs X,\bs o, \bs \mu, (\bs x_n)_n]$. By Theorem~\ref{thm:randomwalk}, $\hat{\nu}\in\mathcal U'$. Therefore, by the ergodic decomposition for the action of $\Gamma$ (see e.g., \cite{Fa62}), there exists a probability measure $\mathcal \alpha$ on $\mathcal E'$ such that $\hat{\nu} = \int_{\mathcal E'} \varphi d\alpha(\varphi)$. This implies that $\nu = \int_{\mathcal E'}(\pi_*\varphi) d\alpha(\varphi)$. Use the change of variables $\psi:=\pi_*\varphi$ and note that $\pi_*\varphi\in\mathcal E$. One gets $\nu = \int_{\mathcal E} \psi d\lambda(\psi)$, where $\lambda = (\pi_*)_*\alpha$. So, the existence of ergodic decomposition is proved.
\end{proof}

\begin{lemma}
	\label{lem:ergodicDecomposition-b}
	For every event $C\subseteq\mstar$ and every $0\leq p\leq 1$, there exists an invariant event $A\in I$ such that $\{\nu\in\mathcal E: \nu(C)\leq p \} = \{\nu\in\mathcal E: \nu(A)=1 \}$.
\end{lemma}
\unwritten{Equivalently, sample intensity of $A$ can be defined simultaneously for all unimodular measures.}

\begin{proof}
	Let $C':=\{(X,o,\mu,(x_n)):(X,o,\mu)\in A\}\subseteq\mathcal M_{\infty}$. By Theorems~1 and~3 of~\cite{Fa62}, there exists $B\in I'$ such that $\{\nu'\in \mathcal E': \nu'(C')\leq p \} = \{\nu'\in \mathcal E': \nu'(B)=1 \}$. Now, let $A$ be the event given by Theorem~\ref{thm:ergodicity-events}.
\end{proof}

\begin{proof}[Proof of Theorem~\ref{thm:ergodicDecomposition} (Uniqueness)]
	Let $\nu_1=\int_{\mathcal E} \varphi d\lambda_1(\varphi)$ and $\nu_2= \int_{\mathcal E} \varphi d\lambda_2(\varphi)$, where $\lambda_1$ and $\lambda_2$ are distinct probability measures on $\mathcal E$. The sets $E_{C,p}:=\{\nu\in\mathcal E: \nu(C)\leq p \}$, where $C\subseteq\mstar$ and $0\leq p\leq 1$, generate the weak topology and its corresponding Borel sigma-field. Therefore, there exists $C$ and $p$ such that $\lambda_1(E_{C,p}) \neq \lambda_2(E_{C,p})$. By Lemma~\ref{lem:ergodicDecomposition-b}, there exists an invariant event $A\in I$ such that $E_{C,p} = \{ \nu\in\mathcal E': \nu(A)=1\}$. Since $\varphi(A)\in\{0,1\}$, one gets $\nu_i (A) = \int \varphi(A)d\lambda_i(\varphi)= \lambda_i(E_{C,p})$. Thus, $\nu_1(A)\neq \nu_2(A)$, and hence, $\nu_1\neq \nu_2$.
\end{proof}

Lemma~\ref{lem:ergodicDecomposition-b} proves condition (b) of~\cite{Fa62}. Condition (c) is also implied by the proof of existence in Theorem~\ref{thm:ergodicDecomposition}. One can similarly prove condition (a) of~\cite{Fa62} using random walks. Therefore, one can leverage the direct construction of the ergodic decomposition in~\cite{Fa62} to prove the following formulation of ergodic decomposition (see the paragraph of~\cite{Fa62} after defining the condition (c)). 

\begin{theorem}
	\label{thm:ergodicDecomposition2}
	There exists an $I$-measurable map $\beta:\mstar\to \mathcal E$ such that $\nu=\int_{\mstar} \beta(\xi)d\nu(\xi)$ for every unimodular probability measure $\nu$ on $\mstar$. In addition, every ergodic $\varphi\in\mathcal E$ is concentrated on $\beta^{-1}(\varphi)$ and no other ergodic measure is concentrated on $\beta^{-1}(\varphi)$. Such a map $\beta$ is unique in the sense that two such maps are equal $\nu$-a.s., for every unimodular probability measures $\nu$.
\end{theorem}
See Theorem~4.11 of~\cite{Ke19} for a similar statement for countable Borel equivalence relations.


\subsection{Ends}


Ends are defined for every topological space~\cite{Fr31}. In particular, if $X$ is boundedly-compact and \textit{locally-connected} and $o\in X$, every end of $X$ can be represented uniquely by a sequence $(U_n)_n$, where $U_n$ is an unbounded connected component of $\oball{n}{o}^c$ and $U_1\supseteq U_2\supseteq\cdots$. There is also a natural compact topology on the union of $X$ and the set of its ends.

For unimodular graphs, it is proved that the number of ends (defined similarly) is either 0, 1, 2 or $\infty$. In addition, in the last case, there is no isolated end (Proposition~6.10 of~\cite{processes}). Here, we extend these results to unimodular \rmm spaces. 

\begin{definition}
	Let $(X,o,\mu)$ be a locally-connected \rmm space. An end of $X$, represented by a sequence $(U_n)_n$ as above, is \defstyle{light} if $\mu(U_n)<\infty$ for some $n$. Otherwise, it is called \defstyle{heavy}. It is called \defstyle{isolated} if $U_n$ has only one end for some $n$.
\end{definition}

If $\bs \mu(\bs X)<\infty$, there is no heavy end, but the number of light tails can be arbitrary (Example~\ref{ex:finitemeasure}). Also, even if $\bs \mu(\bs X)=\infty$, there might exist light isolated ends (e.g., when $\bs X=\cup_{n\in\mathbb Z} \left(\{n\}\times \mathbb R\right) \cup \left(\mathbb R\times \{0\}\right)$ and $\bs \mu$ is the sum of the Lebesgue measure on $\mathbb R\times \{0\}$ and some finite measures on the vertical lines). Here, we focus on the other cases.

\begin{proposition}
	\label{prop:ends}
	Assume $[\bs X, \bs o, \bs \mu]$ is a unimodular \rmm space that is connected and locally-connected a.s. and assume $\bs \mu(\bs X)=\infty$ a.s. 
	\begin{enumerate}[(i)]
		\item The number of heavy ends of $\bs X$ is either 1, 2 or $\infty$. 
		In the last case, there is no isolated heavy end.
		\item The set of light ends is either empty or is an open dense subset of the set of ends of $\bs X$ (and hence, is infinite).
	\end{enumerate}
\end{proposition}

If $\bs X$ is disconnected but locally-connected, the same result holds for almost every connected component of $\bs X$ by Lemma~\ref{lem:subset} (note that every component is clopen by local connectedness).

The proof of the claim for heavy ends is a modification of that of Proposition~3.9 of~\cite{LySc99} for unimodular graphs. So this part is only sketched for brevity. 
First, we provide the following generalization of Lemma~\ref{lem:subset}.

\begin{lemma}
	\label{lem:ends}
	Let $[\bs X, \bs o, \bs \mu]$ be as in Proposition~\ref{prop:ends} and $\bs S$ be a factor subset. Then, almost surely, if $\bs S\neq\emptyset$, then every heavy end of $(\bs X, \bs \mu)$ is a limiting point of $\bs S$.
\end{lemma}
\begin{proof}
	If not, with positive probability, there exists a bounded subset $C$ and a component $U$ of $\bs X\setminus C$ such that $C\cap \bs S\neq \emptyset$, $U\cap \bs S=\emptyset$ and $\bs \mu(U)=\infty$. One might assume that $C$ is open and connected and has diameter at most $n$ (for some fixed $n$). Let $C'$ be the union of all such sets $C$. Then, $C'$ is a factor subset, intersects $\bs S$ and one can see that $\bs X\setminus C'$ has some components $U'$ with infinite mass. For every such $U'$, send unit mass from every $x\in U'$ to the compact set $\partial U'$ (or if $\bs \mu(\partial U')=0$, to some neighborhood of $\partial U'$). This contradicts the MTP since the outgoing mass is at most 1 and the incoming mass is $\infty$ at some points.
\end{proof}

\begin{proof}[Proof of Proposition~\ref{prop:ends}]
	The existence of heavy tails is proved by constructing $U_n$'s inductively such that $\bs \mu(U_n)=\infty$\unwritten{ (notice that, due to local connectedness, by removing $\oball{n}{\bs o}$, only finitely many unbounded component can be added)}.
	If the number of heavy ends is finite and at least three, then one can construct a compact subset $S\subseteq \bs X$, as a factor of $\bs X$, that separates all of the heavy ends (consider the union of all open connected subsets with diameter at most $n$ that separate all ends, and note that every two such subsets intersect).
	This contradicts Lemma~\ref{lem:subset}. 
	Also, if $e$ is an isolated heavy end and there are at least three heavy ends, one can assign to $e$ a bounded set $S(e)\subseteq X$ \textit{equivariantly} that separates $e$ from all other heavy ends and separates at least two other heavy ends from each other (consider the union of all such sets that are open and connected and have diameter at most $n$ and note that every two of them intersect). The set $\cup_e S(e)$ violates Lemma~\ref{lem:ends}. 
	If the set of light ends is nonempty and non-dense, there exists some open connected set $C$ with diameter at most $n$ (for some fixed $n$) such that some component of $\bs X\setminus C$ has only heavy ends, and some other component $U'$ of $\bs X\setminus C$ has light ends and $\bs \mu(U')\leq 1$. One can see that the union of all such $C$ violates Lemma~\ref{lem:ends}.
\end{proof}

\subsection{Amenability}
\label{subsec:amenable}

The notion of amenability is originally defined for countable groups. It is also extended to locally-compact topological groups and also to countable Borel equivalence relations. The latter is extended to unimodular graphs in~\cite{processes}. There are various equivalent definitions, some functional-analytic ones and some combinatorial ones. To extend them to unimodular \rmm spaces, the existing definitions do not apply directly, since no group or countable Borel equivalence relation is present. 

In this subsection, we extend some of the definitions of amenability of countable Borel equivalence relations (in~\cite{CoFeWi81} and~\cite{Ka97}) to unimodular \rmm spaces. 
In Theorem~\ref{thm:amenable}, we prove that these definitions are equivalent by reducing them to analogous conditions for some specific countable Borel equivalence relation. This reduction requires Palm theory developed in Section~\ref{sec:Palm}. So the proof of the main result is postponed to Subsection~\ref{subsec:amenable-proofs}.

\subsubsection{Definition by Local Means}

A definition of amenability of groups is the existence of an \textit{invariant mean}. That is, the existence of a map that assigns a \textit{mean value} to every bounded measurable function such that the map is group-invariant and finitely-additive. For countable Borel equivalence relations, two definitions of global mean and local mean are provided (see~\cite{Ka97} and~\cite{CoFeWi81}). Here, we extend the local mean to unimodular \rmm spaces (global means are based on \textit{partial bijections} and seem more difficult to be extended to the continuum setting).

\begin{definition}[Local Mean]
	Let $[\bs X, \bs o, \bs \mu]$ be a unimodular discrete space. A \defstyle{local mean} $m$ is a map that assigns to (some class of) deterministic \rmm spaces $(X,o,\mu)$, a \textit{state} $m_{(X,o,\mu)}=:m_o$ on $(X,\mu)$ (i.e.,  $m_o:L^{\infty}(X,\mu)\to\mathbb R$ is a positive linear functional such that $m_o(1)=||m_o||_{\infty}=1$), such that:
	\begin{enumerate}[(i)]
		\item $m$ is isomorphism-invariant and is defined for a.e. realization of $[\bs X, \bs o, \bs \mu]$,
		\item If $m_o$ is defined, then $m_y$ is defined for all $y\in X$ and $m_o=m_y$,
		\item For all bounded measurable functions $f:\mdoublestar\to\mathbb R$, the map $[X,o,\mu]\mapsto m_o(f(o,\cdot))$ is measurable.
	\end{enumerate}
%
%
%
%
%
%
\end{definition}

Then, the following condition is a definition of amenability of a unimodular \rmm space $[\bs X, \bs o, \bs \mu]$ and is analogous to the condition~(AI) of~\cite{Ka97}:

\begin{equation}
	\tag{LM}\label{LM}
	\text{There exists a local mean.}
\end{equation}

\subsubsection{Definition by Approximate Means}

\begin{definition}
	Let $[\bs X, \bs o, \bs \mu]$ be a unimodular \rmm space. An \defstyle{approximate mean} is a sequence of measurable functions $\lambda_n:\mdoublestar\to\mathbb R^{\geq 0}$ such that, almost surely, $\forall y\in \bs X: \int_{\bs X}\lambda_n(y,\cdot)d\bs{\mu}=1$ and $\forall y\in \bs X: ||\lambda_n(\bs o, \cdot)-\lambda_n(y,\cdot)||_1\to 0$.
\end{definition}
Here, $\lambda_n(y,\cdot)$ is regarded as an element of $L^1(\bs X, \bs \mu)$. Then, the following condition is another definition of amenability and is analogous to the condition (AI) of~\cite{Ka97} for Borel equivalence relations:

\begin{equation}
	\tag{AM}\label{AM}
	\text{There exists an approximate mean.}
\end{equation}

The name \textit{approximate mean} come from the fact that a local mean is obtained by taking an \textit{ultra limit} of $\int \lambda_n(o,\cdot)f(\cdot)d\mu$ as $n\to\infty$.

\subsubsection{Definition by Hyperfiniteness}

Roughly speaking, a countable Borel equivalence relation is hyperfinite if it can be approximated by finite equivalence sub-relations. The following is analogous to equivalence sub-relations of the natural equivalence relation on $\mstar$ (Subsection~\ref{subsec:Borel}):

\begin{definition}
	A \defstyle{factor partition} $\Pi$ is a map that assigns to every $(X,\mu)$ a partition of $X$ such that the map is invariant under isomorphisms and, if $\Pi(o)$ denotes the element containing $o\in X$, then $\{(X,o,y,\mu): y\in \Pi(o)\}$ is a measurable subset of $\mdoublestar$. It is called \defstyle{finite} if every element of $\Pi$ has finite measure under $\mu$ (but need not be a finite or bounded set). A \defstyle{factor sequence of nested partitions} $(\Pi_n)_n$ is defined similarly by the condition that the map $(X,o,y,\mu)\mapsto \min\{n: y\in \Pi_n(o)\}$ is measurable.
\end{definition}

If $(X,\mu)$ has nontrivial automorphisms, not all partitions of $X$ can appear in the above definition. So, we allow extra randomness. But due to topological issues to define a \textit{random partition} on $X$ (see Subsection~\ref{subsec:extra}), we use another form of extra randomness as follows.

\begin{definition}
	An \defstyle{equivariant (random) partition} is defined similarly to factor partitions with the difference that the partition can be a factor of $(X,o,\Phi)$, where $\Phi$ is an equivariant random additional structure as in Subsection~\ref{subsec:extra}. An \defstyle{equivariant nested sequence of partitions} is defined similarly.
\end{definition}

\begin{remark}
	In fact, the arguments in Subsection~\ref{subsec:amenable-proofs} show that it is enough that the partition is a factor of the \textit{marked Poisson point process} (see Example~\ref{ex:poisson} and Subsection~\ref{subsec:amenable-proofs}). Also, if there is no nontrivial automorphism a.s., then factor partitions are enough.
\end{remark} 

These definitions allow us to define the following three forms of hyperfiniteness, which will be seen to be equivalent.

\begin{align}
	&&\nonumber\text{There exist equivariant nested finite partitions } \Pi_n\\ 
	\tag{HF1}\label{HF1}&& \text{such that } \myprob{\bigcup_n \Pi_n(\bs o)=\bs X}=1.
\end{align}

This does not imply that $\Pi_n(\bs o)$ contains a large neighborhood of $\bs o$. The following condition is another form of hyperfiniteness.

\begin{align}
	&&\nonumber\text{There exist equivariant nested finite partitions } \Pi_n \\ 
	\tag{HF2}\label{HF2}&& \text{such that }\forall r<\infty: \myprob{\exists n: B_r(\bs o)\subseteq \Pi_n(\bs o)}=1.
\end{align}

Hyperfiniteness can also be phrased in terms of a single partition as follows.

\begin{align}
	\nonumber\forall r<\infty, \forall \epsilon>0, \text{ there exists an equivariant finite partition } \Pi\\
	 \text{such that } \myprob{B_r(\bs o)\not\subseteq \Pi(\bs o)}<\epsilon. 
	\tag{HF3}\label{HF3}
\end{align}

\subsubsection{Definition by The Folner Condition}

The Folner condition is a combinatorial way to define amenability of groups (and deterministic graphs); i.e., the existence of a finite set $A$ such that the \textit{boundary} of $A$ is arbitrarily small compared to $A$. Modified versions of this condition are provided for countable Borel equivalence relations (based on equivariant graphings; see~\cite{CoFeWi81} and~\cite{Ka97}) and for unimodular graphs (Definition~8.1 of~\cite{processes}). In particular, $A$ is required to be an element of some factor partition. The extension to the continuum setting is not straightforward. Here, we provide two Folner-type conditions as follows.
For $A\subseteq \bs X$ and $r\geq 0$, let $\partial_r A:=\{y\in A: B_r(y)\not\subseteq A \}$ denote the \textit{inner $r$-boundary} of $A$. 

\begin{align}
	\nonumber\forall r<\infty, \forall \epsilon>0, \text{ there exists an equivariant finite partition } \Pi\\
	\text{such that } \omid{\frac{\bs \mu(\partial_r\Pi(\bs o))}{\bs \mu(\Pi(\bs o))}}<\epsilon. 
	\tag{FO1}\label{FO1}
\end{align}

The MTP implies easily that this condition is equivalent to~\eqref{HF3} (see the proof of Theorem~\ref{thm:amenable}).

\begin{align}
	&&\nonumber\text{There exist equivariant nested finite partitions } \Pi_n \text{ such that}\\ 
	\tag{FO2}\label{FO2}&& \forall r: \frac{\bs \mu(\partial_r\Pi_n(\bs o))}{\bs \mu(\Pi_n(\bs o))}\to 0,\quad a.s.
\end{align}

It is not clear to the author whether one can use the \textit{outer boundary} or the full \textit{boundary} in the above definitions or not.





\subsubsection{Equivalence of the Definitions}

The following is the main result of this subsection.

\begin{theorem}
	\label{thm:amenable}
	For unimodular \rmm spaces, the conditions~\eqref{LM}, \eqref{AM}, \eqref{HF1}, \eqref{HF2}, \eqref{HF3}, \eqref{FO1} and~\eqref{FO2} are equivalent. 
\end{theorem}

\begin{definition}[Amenability]
	A unimodular \rmm space is called \defstyle{amenable} if the equivalent conditions in Theorem~\ref{thm:amenable} hold.
\end{definition}

The proof of the analogous results for countable groups and for countable Borel equivalence relations (Theorem~1 of~\cite{Ka97}) is heavily based on the discreteness.
So, these proofs cannot be directly extended to unimodular \rmm spaces. We will prove the above theorem by reducing it to the analogous result for an specific countable Borel equivalence relation. The reduction is based on Palm theory developed in Section~\ref{sec:Palm} and the proof is postponed to Subsection~\ref{subsec:amenable-proofs}. In short, a countable Borel equivalence relation is constructed using the \textit{marked Poisson point process} on $\bs X$ (Example~\ref{ex:poisson}). An invariant measure is constructed by the Palm distribution (Example~\ref{thm:PalmofPoisson}). Then, the reduction is proved using the \textit{Voronoi tessellation} and \textit{balancing transport kernels} (later: cross ref).
	
Here, we only prove the implications that do not rely on Section~\ref{sec:Palm}:

\begin{lemma}
	\label{lem:hyperfinite}
	 \eqref{HF2} $\Rightarrow$ \eqref{HF3} $\Rightarrow$ \eqref{HF1} and \eqref{HF3} $\Leftrightarrow$ \eqref{FO1} $\Leftrightarrow$ \eqref{FO2}.
%
\end{lemma}
\begin{proof}
	The implication \eqref{HF2} $\Rightarrow$ \eqref{HF3} is clear.
	
	\eqref{HF3} $\Rightarrow$ \eqref{HF1}. Let $\Pi_n$ be the equivariant partition given by~\eqref{HF3} for $r:=n$ and $\epsilon:=2^{-n}$. Then, one can let $\Pi'_n$ be the superposition of $\Pi_n,\Pi_{n+1},\ldots$ and use the Borel Cantelli lemma to deduce~\eqref{HF1}.
	
	\eqref{HF3} $\Leftrightarrow$ \eqref{FO1}. By letting $g(o, y):= 1/{\mu(\Pi(o))}\identity{\{y\in \Pi(o)\}}\identity{\{B_r(o)\not\subseteq \Pi(o) \}}$, the MTP gives that $\myprob{B_r(\bs o)\not\subseteq \Pi(\bs o)} = \omid{\bs \mu(\partial_r\Pi(\bs o))/\bs \mu(\Pi(\bs o))}$. 
	
	\eqref{FO1} $\Leftrightarrow$ \eqref{FO2}. The $\Rightarrow$ part is obtained by taking a subsequence such that the almost sure convergence holds. The $\Leftarrow$ part is obtained by the bounded converge theorem noting that $\partial_r A\subseteq A$.
	%
%
%
%
%
%
%
%
%
%
\end{proof}

\begin{remark}
	Theorem~1 of~\cite{Ka97} assumes ergodicity, but the claim holds for non-ergodic cases as well. In fact, in each of the definitions, $[\bs X, \bs o, \bs \mu]$ is amenable if and only if almost all of its ergodic components are amenable.
\end{remark}

\section{Palm Theory on Unimodular \rmm Spaces}
\label{sec:Palm}

Recall from Subsection~\ref{subsec:pointprocess} that the Palm version of a stationary point process (or random measure) is defined and means heuristically re-rooting to a \textit{typical point} of the point process.
In this section, we generalize Palm theory to unimodular \rmm spaces. Here, a unimodular \rmm space $[\bs X, \bs o, \bs \mu]$ is thought of as a generalization of the Euclidean or hyperbolic spaces. The Palm theory is intended for random measures $\Phi$ on $\bs X$ which are chosen \textit{equivariantly} (e.g., the Poisson point process with intensity measure $\bs \mu$). The notion of equivariant random measures is discussed in Subsection~\ref{subsec:equivmeasure} and the Palm theory is given in the next subsections.

\subsection{Additional Random Measures on a Unimodular Random \rmm Space}
\label{subsec:equivmeasure}
For $k\geq 1$, let $\mkstar{k}$ be the space of all tuples $(X,o,\mu_1,\ldots,\mu_k)$, where $(X,o,\mu_1)$ is a \rmm space and $\mu_2,\ldots,\mu_k$ are measures on $X$. Likewise, let $\mkdoublestar{k}$ be the space of all doubly-rooted tuples $(X,o_1,o_2,\mu_1,\ldots,\mu_k)$. By the discussion in Subsection~\ref{subsec:extra} and using the results of~\cite{Kh19generalization}, one can equip $\mkstar{k}$ and $\mkdoublestar{k}$ with generalizations of the GHP metric such that they are Polish spaces.

Let $[\bs X, \bs o, \bs \mu]$ be a unimodular \rmm space. Roughly speaking, an \textit{equivariant random measure on $\bs X$} is assigning an additional measure $\Phi$ on $\bs X$ (in every realization of $[\bs X, \bs o, \bs \mu]$), possibly using extra randomness, in a way that unimodularity is preserved. Intuitively, the law of $\Phi$ (given $[\bs X, \bs o, \bs \mu]$) should not depend on the root, should be isomorphism-invariant and satisfy some measurability condition. More precisely:

\begin{definition}
	\label{def:equivmeasure}
	An \defstyle{equivariant random measure} $\Phi$ is a map that assigns to every deterministic \rmm space $(X,o,\mu)$ a random measure $\Phi_{(X,o,\mu)}$ on $X$ such that:
	\begin{enumerate}[(i)]
		\item \label{def:equivmeasure:i} For all $(X,o,\mu)$ and all $y\in X$, $\Phi_{(X,y,\mu)}\sim \Phi_{(X,o,\mu)}$,
		\item If $\rho:(X,o,\mu)\to(X',o',\mu')$ is an isomorphism, then $\rho_*\Phi_{(X,o,\mu)}\sim \Phi_{(X',o',\mu')}$,
		\item For all Borel subsets $A\subseteq \mkstar{2}$, the map $(X,o,\mu)\mapsto \myprob{[X,o,\mu,\Phi_{(X,o,\mu)}]\in A}$ is measurable.
	\end{enumerate}
	In addition, if $[\bs X, \bs o, \mu]$ is a unimodular \rmm space, an \defstyle{equivariant random measure on $\bs X$} is a map with the above conditions relaxed to be defined on an invariant event with full probability. We denote $\Phi_{(\bs X, \bs o, \bs \mu)}$ simply by $\Phi$ for brevity.
	
	As usual, if for all $(X,o,\mu)$, $\Phi_{(X,o,\mu)}$ is a counting measure a.s., $\Phi$ is called an \defstyle{equivariant (simple) point process} and if $\Phi_{(X,o,\mu)}(\cdot)\in\mathbb Z$ a.s., it is called an \defstyle{equivariant (non-simple) point process}.
\end{definition}

By integrating the distribution of $\Phi$ over the distribution of $[\bs X, \bs o, \bs \mu]$, one obtains a probability measure on $\mkstar{2}$ (similarly to (2.3) of~\cite{I}). This determines a random object of $\mkstar{2}$. By an abuse of notation, we denote the latter by $[\bs X, \bs o, \bs \mu, \Phi]$ and we use the same symbols $\mathbb P$ and $\mathbb E$ for its distribution. It is easy to deduce that $[\bs X, \bs o, \bs \mu, \Phi]$ is unimodular in the sense of Subsection~\ref{subsec:extra} (see Lemma~\ref{lem:equivmeasure}); i.e., 


\begin{definition}
	A random element $[\bs Y, \bs p, \bs \mu_1, \bs \mu_2]$ on $\mkstar{2}$ is \defstyle{unimodular} if the following MTP holds for all measurable functions $g:\mkdoublestar{2}\to\mathbb R^{\geq 0}$:
	\[
		\omid{\int_{\bs Y} g(\bs p, z)d\bs\mu_1(z)} = \omid{\int_{\bs Y} g(z,\bs p)d\bs\mu_1(z)}.
	\]
	Note that the integral is against $\bs \mu_1$, but $g$ can depend on $\bs \mu_2$ as well.
\end{definition}

One can also extend the above definition to define a \defstyle{jointly equivariant} pairs (or tuples) of random measures $(\Phi_{(X,o,\mu)},\Psi_{(X,o,\mu)})$ and the results of this section remain valid.

\subsubsection{An Equivalent Definition}
\label{subsec:equivmeas2}

A simpler definition of equivariant random measures on $\bs X$ would be a unimodular tuple $[\bs Y, \bs p, \bs \mu_1, \bs \mu_2]$ such tat $[\bs Y, \bs p, \bs \mu_1]$ has the same distribution as $[\bs X, \bs o, \bs \mu]$. Indeed, the two definitions are equivalent in the following sense.

\begin{lemma}
	\label{lem:equivmeasure}
	Let $[\bs X, \bs o, \bs \mu]$ be a nontrivial unimodular \rmm space.
	If $\Phi$ is an equivariant random measure on $\bs X$, then $[\bs X, \bs o, \bs \mu, \Phi]$ is unimodular. Conversely, if $[\bs Y, \bs p, \bs \mu_1, \bs \mu_2]$ is unimodular and $[\bs Y, \bs p, \bs \mu_1]\sim [\bs X, \bs o, \bs \mu]$, then there exists an equivariant random measure $\Phi$ such that $[\bs Y, \bs p, \bs \mu_1, \bs \mu_2]\sim [\bs X, \bs o, \bs \mu, \Phi]$.
\end{lemma}

The proof of this lemma is based on the Palm theory developed in the next subsections and will be given in Subsection~\ref{subsec:disintegration} (the lemma will not be used in this section).
To prove the converse, one basically needs to consider the regular conditional distribution of $[\bs Y, \bs p, \bs \mu_1, \bs \mu_2]$ w.r.t. $[\bs Y, \bs p, \bs \mu_1]$. However, a difficult step is proving that a version of the conditional law exists which does not depend on the root (property~\eqref{def:equivmeasure:i} in Definition~\ref{def:equivmeasure}). 
This will be proved using the Palm theory and \textit{invariant disintegration}~\cite{disintegration}.

\begin{remark}
	The above alternative definition of equivariant measures is useful for weak convergence of equivariant random measures and tightness. See the following definition and corollary. It also enables us to regard $\bs \mu$ as an equivariant measure on (the Palm version of) $[\bs X, \bs o, \Phi]$, which will be explained in Subsection~\ref{subsec:Palm-unimodular}. In contrast, Definition~\ref{def:equivmeasure} is useful for explicit constructions and also for dealing with couplings of two equivariant random measures; e.g., to define the \textit{independent coupling} (conditional to $[\bs X, \bs o, \bs \mu]$) of two equivariant random measures (Example~\ref{ex:independentcoupling}).
\end{remark}


\begin{definition}
	Two equivariant random measures $\Phi$ and $\Psi$ on $\bs X$ are equivalent if $[\bs X, \bs o, \bs \mu, \Phi]\sim [\bs X, \bs o, \bs \mu, \Psi]$ (equivalently, on almost every realization of $[\bs X, \bs o, \bs \mu]$, one has $\Phi_{(\cdot)}\sim \Psi_{(\cdot)}$). Also, we say that $\Phi_n$ converges to $\Phi$ if $[\bs X, \bs o, \bs \mu, \Phi_n]$ converges weakly to $[\bs X, \bs o, \bs \mu, \Phi]$.
\end{definition}

\begin{corollary}
	Let $b:\mathcal M_*\times\mathbb N\to\mathbb R$ be a lower semi-continuous function (e.g., $(X,o,\mu,n)\mapsto \mu(B_n(o))$). Then, the set of equivariant random measures $\Phi$ on $\bs X$ such that $\forall n: \Phi(\oball{n}{\bs o})\leq b([\bs X, \bs o, \bs \mu], n)$ a.s. is compact.
\end{corollary}

\begin{remark}
	\label{rem:disintegration}
	Lemma~\ref{lem:equivmeasure} can be generalized to the case where $\Phi$ is any other type of additional random structures on $\bs X$, discussed in Subsection~\ref{subsec:extra}, as long as the Polishness property holds. The same proof works in the general case, but for simplicity, we provide the proof for random measures only.
\end{remark}

\subsubsection{Examples}

For basic examples of equivariant random measures, one can mention $\Phi=0$, $\Phi=\mu$, $\Phi=b\mu$ (see~\eqref{eq:b mu}), or more generally, any \defstyle{factor measures}; i.e., when $\Phi_{(X,o,\mu)}$ is a deterministic function of $(X,\mu)$ satisfying the assumptions in Definition~\ref{def:equivmeasure}. The following are further examples.

\begin{example}[Stationary Random Measure]
	\label{ex:euclidean-equivariant}
	When $[\bs X, \bs o, \bs \mu]:=[\mathbb R^d, 0, \mathrm{Leb}]$, every stationary point process or random measure on $\mathbb R^d$ is an equivariant random measure on $\bs X$ in the sense of Subsection~\ref{subsec:equivmeas2}. However, in order to have the conditions in Definition~\ref{def:equivmeasure}, one should  assume that it is isometry-invariant (or just apply a random isometry fixing 0). By the modification mentioned in Remark~\ref{rem:lostaxis}, one can say that stationarity is equivalent to being an equivariant random measure on $\mathbb R^d$.
\end{example}

\begin{example}[Intensity Measure]
	If $\Phi$ is an equivariant random measure, then the \defstyle{intensity measure} of $\Phi$, defined by $\Psi_{(X,o,\mu)}(\cdot):= \omid{\Phi_{(X,o,\mu)}(\cdot)}$, is also equivariant. In fact, the intensity measure is a {factor measure}.
\end{example}

\begin{example}[Poisson Point Process]
	\label{ex:poisson}
	Let $\Phi$ be an equivariant random measure. The \defstyle{Poisson point process} with intensity measure $\Phi$ is an equivariant random measure (defined by considering in every realization of $\Phi_{(X,o,\mu)}$, the classical definition of the Poisson point process with intensity measure $\Phi_{(X,o,\mu)}$). In addition, $\Phi$ and the Poisson point process are jointly-equivariant. Note that if $\Phi$ has atoms, the the Poisson point process has multiple points and is not a simple point process. 
	\\
	One can prove that if $[\bs X, \bs o, \bs \mu]$ is ergodic and $\bs \mu(\bs X)=\infty$ a.s., then $[\bs X, \bs o, \bs \mu, \Phi]$ is also ergodic (this fails if $0<\bs \mu(\bs X)<\infty$). The proof is similar to Lemma~4 of~\cite{shift-coupling} and is skipped.
\end{example}

\begin{example}[Independent Coupling]
	\label{ex:independentcoupling}
	Let each of $\Phi$ and $\Psi$ be an equivariant random measure on $\bs X$. Using the definition of Subsection~\ref{subsec:equivmeas2}, it is not trivial to define a coupling of $\Phi$ and $\Psi$, but it is easy using Definition~\ref{def:equivmeasure}: For every deterministic $(X,o,\mu)$, consider an independent coupling of $\Phi_{(X,o,\mu)}$ and $\Psi_{(X,o,\mu)}$. Then, $(\Phi,\Psi)$ becomes jointly-equivariant and is called the \defstyle{independent coupling} (conditional to $[\bs X, \bs o, \bs \mu]$) of $\Phi$ and $\Psi$.
\end{example}

\subsection{Palm Distribution and Intensity}
\label{subsec:Palm}


Let $[\bs X, \bs o, \bs \mu]$ be a nontrivial unimodular \rmm space and $\Phi$ be an equivariant random measure on $\bs X$ (in the sense of either Definition~\ref{def:equivmeasure} or Subsection~\ref{subsec:equivmeas2}). Inspired by the analogous notions in stochastic geometry (Subsection~\ref{subsec:pointprocess}), we are going to define the Palm distribution and the intensity. The approach~\eqref{eq:PalmEuclidean} does not work here since the base space is not fixed. We will extend two other approaches by using the Campbell measure and tessellations.

The \defstyle{Campbell measure} $C_\Phi$ is the measure on $\mkdoublestar{2}$ defined by
\begin{equation}
	\nonumber
	C_{\Phi}(A):=\omid{\int_{\bs X}\identity{A}(\bs o, y) d\Phi(y)},
\end{equation}
where, as before, we use the abbreviation $\identity{A}(\bs o, y)$ for $\identity{A}(\bs X, \bs o, y, \bs \mu, \Phi)$. It is straightforward to see that $C_{\Phi}$ is a sigma-finite measure. The Palm distribution will be obtained by a disintegration of the Campbell measure as follows.

\begin{theorem}
	\label{thm:Palm}
	There exists a unique sigma-finite measure $Q_{\Phi}$ on $\mkstar{2}$ such that for all measurable functions $g:\mkdoublestar{2}\to\mathbb R^{\geq 0}$,
	\begin{equation}
		\label{eq:Palm1}
		\omid{\int_{\bs X}g(\bs o, y) d\Phi(y)} = \int g d C_{\Phi} = \int_{\mkstar{2}} \left(\int_{X}g(y,o) d\mu(y)\right) dQ_{\Phi}([X,o,\mu,\varphi]).
	\end{equation}
\end{theorem}

Note that the left hand side is just the definition of $C_{\Phi}$.

\begin{definition}[Intensity and Palm Distribution]
	\label{def:Palm}
	The \defstyle{intensity} $\lambda_{\Phi}$ of $\Phi$ (w.r.t. $\bs \mu$) is the total mass of $Q_{\Phi}$. If $0<\lambda_{\Phi}<\infty$, then one can normalize $Q_{\Phi}$ to find a probability measure $\mathbb P_{\Phi}:=\frac 1{\lambda_{\Phi}}Q_{\Phi}$ on $\mkstar{2}$, which is called the \defstyle{Palm distribution} of $\Phi$. 
	\\
	By an abuse of notation, we use the same symbol $[\bs X, \bs o, \bs \mu, \Phi]$ when dealing with $\mathbb P_{\Phi}$; e.g., we use formulas like $\probPalm{\Phi}{[\bs X, \bs o, \bs \mu, \Phi]\in A}$ and $\omidPalm{\Phi}{g(\bs o)}$ by keeping in mind that the probability measure has changed.
\end{definition}

By the above notation, we can rewrite~\eqref{eq:Palm1} as follows:

\begin{equation}
	\label{eq:Campbell}
			\omid{\int_{\bs X}g(\bs o, y) d\Phi(y)} = \lambda_{\Phi} \omidPalm{\Phi}{\int_{\bs X}g(y,\bs o) d\bs \mu(y)}.	
\end{equation}

Example~\ref{ex:euclidean-equivariant} shows that the above definition generalizes the Palm distribution of stationary random measures on $\mathbb R^d$. In addition, \eqref{eq:Campbell} is a generalization of the \defstyle{refined Campbell Theorem} (see~\cite{ThLa09} or~\cite{bookScWe08}).

The intuition of the Palm distribution is that, under $\mathbb P_{\Phi}$, the root is a typical point of $\Phi$. The equation~\eqref{eq:Campbell} means that, if $\bs \mu$ (resp. $\Phi$) is interpreted as a measure on a set of senders (resp. receivers), then the expected mass sent from a typical sender is equal to $\lambda_{\Phi}$ times the expected mass received by a typical receiver.

We will prove Theorem~\ref{thm:Palm} by constructing the Palm distribution directly. The construction~\eqref{eq:PalmEuclidean} does not work since the underlying space $\bs X$ is not deterministic and one cannot fix a subset $B$. We replace it by a transport function $h$ such that $h^-(\cdot)=1$ as follows. This extends Mecke's construction in~(2.8) and Satz~2.3 of~\cite{Me67}. This also extends the construction of a \textit{typical cell} for equivariant tessellations of stationary point processes.


\begin{theorem}[Construction of Palm]
	\label{thm:Palmconstruction}
	Let $h:\mdoublestar\to\mathbb R^{\geq 0}$ be a transport function such that $\forall y\in \bs X: h^-(y)=1$ a.s. Then, for every equivariant random measure $\Phi$, the intensity of $\Phi$ and the measure $Q_{\Phi}$ are obtained as follows and satisfy~\eqref{eq:Palm1}:
	\begin{eqnarray}
		\label{eq:intensity}
		\lambda_{\Phi} &=& \omid{h^+_{\Phi}(\bs o)}:= \omid{\int_{\bs X} h(\bs o, z) d\Phi(z)},\\
		\label{eq:Q_Phi}
		Q_{\Phi}(A) &=& \omid{\int_{\bs X} \identity{A}(\bs X, z, \bs \mu, \Phi)h(\bs o, z) d\Phi(z)}.
	\end{eqnarray}
	So, if $0<\lambda_{\Phi}<\infty$, the Palm distribution is obtained by biasing by $h^+_{\Phi}(\bs o)$, and then, re-rooting to a point chosen with law proportional to $h(\bs o, \cdot) \Phi$; i.e.,
	\begin{equation}
		\label{eq:Palmconstruction}
		\mathbb P_{\Phi}(A) = \frac{1}{\lambda_{\Phi}} \omid{\int_{\bs X} \identity{A}(\bs X, z, \bs \mu, \Phi) h(\bs o, z)d\Phi(z)}.
	\end{equation}
\end{theorem}

Note that such a transport function $h$ exists by Lemma~\ref{lem:balancingkernel}. 
 
\begin{proof}
	It is enough to prove that $Q_{\Phi}$ satisfies~\eqref{eq:Palm1}. Let $g$ be a transport function on $\mkdoublestar{2}$. The right hand side of~\eqref{eq:Palm1} equals to:
	\begin{eqnarray*}
		\int_{\mkstar{2}} g^-(o) dQ_{\Phi}([X,o,\mu,\varphi]) &=&\omid{\int_{\bs X} g^-(z) h(\bs o, z) d\Phi(z)}\\
		&=& \omid{\int_{\bs X} \int_{\bs X} g(y,z) h(\bs o, z)d\Phi(z) d\bs \mu(y) }\\
		&=& \omid{\int_{\bs X} \int_{\bs X} g(\bs o,z) h(y, z)d\Phi(z) d\bs \mu(y) }\\
		&=& \omid{\int_{\bs X} g(\bs o,z) d\Phi(z)},
	\end{eqnarray*}
	where the first equality is by the definition of $Q_{\Phi}$, the third one is by the MTP, and the last one is because $h^-(z)=1$. This proves~\eqref{eq:Palm1}.
\end{proof}
 
\begin{proof}[Proof of Theorem~\ref{thm:Palm}]
	The existence is proved in Theorem~\ref{thm:Palmconstruction}. For uniqueness, assume $Q_{\Phi}$ is a measure satisfying~\eqref{eq:Palm1}. Let $h$ be a transport function such that $h^-(\cdot)=1$ a.s. For an arbitrary event $A\subseteq\mkstar{2}$, let $g(o,z):=h(o,z)\identity{A}(z)$. Inserting $g$ into~\eqref{eq:Palm1} gives $Q_{\Phi}(A) = \omid{\int \identity{A}(y)h(\bs o, y)d\Phi(y)}$; i.e., $Q_{\Phi}$ is the measure given in Theorem~\ref{thm:Palmconstruction}.
\end{proof}

As a first application of the Palm calculus, we prove the following extension of Lemma~\ref{lem:happensatroot}.

\begin{proposition}
	Let $\Phi$ be an equivariant random measure on $\bs X$. If $\bs{\mu}(\bs X)=\infty$ a.s., then $\Phi(\bs X)\in\{0,\infty\}$ a.s.
\end{proposition}
\begin{proof}
	Let $A$ be the event that $\bs \mu(\bs X)=\infty$ and $0<\Phi(\bs X)<\infty$ and assume $\myprob{A}>0$. Since $A$ is an invariant event, \eqref{eq:Q_Phi} implies that $Q_{\Phi}(A)>0$. On $A$, let $g(u,v):=\Phi(\bs X)^{-1}$. Then, $g^+_{\Phi}(\cdot)=1$ and $g^-(\cdot)=\infty$ on $A$. This contradicts~\eqref{eq:Palm1}.
\end{proof}
 
 \subsection{Properties of Palm}
  
 \subsubsection{Unimodularity of the Palm Version}
\label{subsec:Palm-unimodular}
 
 The Palm distribution of an equivariant random measure is generally not unimodular. However, it satisfies the following mass transport principle similarly to the case of stationary random measures (Subsection~\ref{subsec:pointprocess}).
 
 \begin{definition}
 	\label{def:unimodularw.r.t}
	A random element $[\bs Y, \bs p, \bs \mu_1, \bs \mu_2]$ on $\mkstar{2}$ is \defstyle{unimodular with respect to $\bs \mu_2$} if the following MTP holds for all measurable functions $g:\mkdoublestar{2}\to\mathbb R^{\geq 0}$:
 	\[
 	\omid{\int_{\bs Y} g(\bs p, z)d\bs\mu_2(z)} = \omid{\int_{\bs Y} g(z,\bs p)d\bs\mu_2(z)}.
 	\]
 	This is equivalent to the unimodularity of $[\bs Y, \bs p, \bs \mu_2, \bs \mu_1]$, which is obtained by swapping the two measures.
 \end{definition}
 
 \begin{theorem}
 	\label{thm:Palm-unimodular}
 	Let $[\bs X, \bs o, \bs \mu]$ be a unimodular \rmm space and $\Phi$ be an equivariant random measure on $\bs X$ with positive finite intensity. Then, under the Palm distribution $\mathbb P_{\Phi}$, $[\bs X, \bs o, \bs \mu, \Phi]$ is unimodular w.r.t. $\Phi$. In other words, $[\bs X, \bs o, \Phi,\bs{\mu}]$ is unimodular under the Pam distribution.
 \end{theorem}

 
 \begin{proof}
 	Let $g$ be a transport function and $h$ be as in Theorem~\ref{thm:Palmconstruction}. By~\eqref{eq:Palmconstruction},
 	\begin{eqnarray*}
 		\omidPalm{\Phi}{g^+_{\Phi}(\bs o)} &=& \frac{1}{\lambda_{\Phi}} \omid{\int_{\bs X}g^+_{\Phi}(y) h(\bs o, y) d\Phi(y) }\\
 		&=& \frac{1}{\lambda_{\Phi}} \omid{\int_{\bs X} \int_{\bs X} g(y,z) h(\bs o, y) d\Phi(y) d\Phi(z) }\\
 		&=& \frac{1}{\lambda_{\Phi}} \omidPalm{\Phi}{\int_{\bs X} \int_{\bs X} g(y,\bs o) h(z, y) d\Phi(y) d\bs\mu(z) }\\
 		&=& \frac{1}{\lambda_{\Phi}} \omidPalm{\Phi}{\int_{\bs X}  g(y,\bs o) d\Phi(y)},
 	\end{eqnarray*}
 	where in the third equality, we have swapped $\bs o$ and $z$ by the refined Campbell  formula~\eqref{eq:Campbell} and the last equality holds by $h^{-1}(y)=1$. This proves the claim.
 \end{proof}

\subsubsection{Exchange Formula}

Neveu's exchange formula (see Theorem~3.4.5 of~\cite{bookScWe08}) is a form of the MTP between two jointly-stationary random measures on $\mathbb R^d$. The refined Campbell formula~\eqref{eq:Campbell} can be thought of an exchange formula between $\bs \mu$ and $\Phi$. More generally, assume $(\Phi,\Psi)$ is a pair of random measures on $[\bs X, \bs o, \bs \mu]$ which are jointly equivariant. Assume the intensities $\lambda_{\Phi}$ and $\lambda_{\Psi}$ are positive and finite and consider the Palm distributions $\mathbb P_{\Phi}$ and $\mathbb P_{\Psi}$. 

\begin{proposition}[Exchange Formula]
	For jointly-equivariant random measures $\Phi$ and $\Psi$ on $\bs X$ as above, and for all transport functions $g:\mkdoublestar{3}\to\mathbb R^{\geq 0}$, 
	\begin{equation}
		\label{eq:neveu}
		\lambda_{\Phi} \omidPalm{\Phi}{\int_{\bs X}g(\bs o, y) d\Psi(y)} = \lambda_{\Psi} \omidPalm{\Psi}{\int_{\bs X}g(y,\bs o) d\Phi(y)},
	\end{equation}
	where $g(u,v)$ abbreviates $g(\bs X, u, v, \bs \mu, \Phi, \Psi)$ as usual.
\end{proposition}

\begin{proof}
	Let $h$ be as in Theorem~\ref{thm:Palmconstruction}. By~\eqref{eq:Palmconstruction}, the left hand side is equal to
	\begin{eqnarray*}
		\omid{\int \int h(\bs o, z) g(z,y) d\Psi(y) d\Phi(z)} = \omid{\int f(\bs o, y) d\Psi(y)},
	\end{eqnarray*}
	where $f(u,v):=\int_{\bs X} h(u,z)g(z,v)d\Phi(z)$. By \eqref{eq:Campbell}, the last term is equal to 
	\begin{eqnarray*}
		\lambda_{\Psi}\omidPalm{\Psi}{\int f(y,\bs o) d\bs \mu(y)} = \lambda_{\Psi} \omidPalm{\Psi}{\int\int h(y,z)g(z,\bs o)d\Phi(z) d\bs\mu(y)}.
	\end{eqnarray*}
	Since $h^-(z)=1$, the last term is equal to the right hand side of~\eqref{eq:neveu} and the claim is proved.
\end{proof}

Another interpretation of the exchange formula is as follows. By Theorem~\ref{thm:Palm-unimodular}, under the Palm distribution of $\Phi$, $[\bs X, \bs o, \bs \mu, \Phi, \Psi]$ is unimodular w.r.t. $\Phi$. Now, one may regard $\Psi$ as a random measure on the latter and think of $\bs \mu$ as a decoration. Then, the exchange formula~\eqref{eq:neveu} turns into the refined Campbell formula for $\Psi$ with respect to $\Phi$. More precisely, one can deduce the following.

\begin{corollary}
	\label{cor:exchange}
	Under the Palm distribution of $\Phi$, and by thinking of $\Phi$ as the base measure on $\bs X$, the intensity of $\Psi$ would be $\lambda_{\Psi}/\lambda_{\Phi}$ and the Palm distribution of $\Psi$ would be identical to $\mathbb P_{\Psi}$.
\end{corollary}

\subsubsection{Palm Inversion = Palm}
\label{subsec:Palminversion}

\textit{Palm inversion} refers to the construction of the stationary distribution from the Palm distribution. Similarly, we desire to reconstruct the distribution of $[\bs X, \bs o, \bs \mu, \Phi]$ from the Palm distribution $\mathbb P_{\Phi}$. 

By Theorem~\ref{thm:Palm-unimodular}, $[\bs X, \bs o, \Phi, \bs \mu]$ is unimodular under $\mathbb P_{\Phi}$. Now, $\bs \mu$ can be regarded as a random measure (in the sense of Subsection~\ref{subsec:equivmeas2}) on the Palm version of $[\bs X, \bs o, \Phi]$. Therefore, it makes sense to speak of the Palm distribution of $\bs \mu$, namely, $\mathbb P'$. Using the symmetry of $\bs \mu$ and $\Phi$ in the refined Campbell formula~\eqref{eq:Campbell}, it is straightforward to deduce that $\mathbb P'$ is equal to the distribution of $[\bs X, \bs o, \bs \mu, \Phi]$ (see Corollary~\ref{cor:exchange}). An explicit construction of $\mathbb P'$ can also be provided by Theorem~\ref{thm:Palmconstruction}.

The above discussion shows that the generalization of Palm theory in this section unifies Palm and Palm-inversion as well. In particular, if $\Phi_0$ is the Palm version of a stationary point process in $\mathbb R^d$, the stationary version is obtained by considering the Palm distribution of the Lebesgue measure w.r.t. $\Phi_0$. Indeed, for Palm inversion, one desires to move the origin to a typical point of the Euclidean space.


\subsubsection{Stationary Distribution of Random Walks}
\label{subsec:Palm-randomwalk}
In Theorem~\ref{thm:randomwalk}, we considered random walks on $[\bs X, \bs o, \bs \mu]$ such that the Markovian kernel preserves a measure of type $b\bs \mu$ a.s. If the Markovian kernel has a stationary distribution which is not absolutely continuous w.r.t. $\bs \mu$, one can extend this result as follows.

Let $\Phi$ be an equivariant random measure with positive and finite intensity on  a unimodular \rmm space $[\bs X, \bs o, \bs \mu]$. Let $k$ a Markovian kernel (that might depend on $\Phi$ as well) and $(\bs x_n)_n$ be the random walk with kernel $k$ as in Theorem~\ref{thm:randomwalk}. Define a shift operator $\mathcal S$ similarly to~\eqref{eq:shift}. Since $[\bs X, \bs o, \bs \mu, \Phi, (\bs x_n)_n]$ is unimodular, one can define the Palm distribution of $\Phi$ on it, which we call $Q$.


\begin{proposition}
	In the above setting, if in almost every sample $[X,o,\mu,\varphi]$ of $[\bs X, \bs o, \bs \mu, \Phi]$, $\varphi$ is a stationary (resp. reversible) measure for the Markovian kernel, then the Palm distribution $Q$ of $\Phi$ is a stationary (resp. reversible) measure for the shift operator $\mathcal S$. 
\end{proposition}
\begin{proof}
	By Theorem~\ref{thm:Palm-unimodular}, the claim is reduced to Theorem~\ref{thm:randomwalk}.
\end{proof}

\subsection{Examples}

We already mentioned that the Palm distribution defined in this section generalizes the case of stationary point processes and random measures (by Example~\ref{ex:euclidean-equivariant}). The following are further examples of the Palm distribution.

\begin{example}[Conditioning]
	Let $\bs S$ be a factor subset (Definition~\ref{def:subset}) and let $\Phi:=\restrict{\bs \mu}{\bs S}$. Then, the intensity of $\Phi$ is $\myprob{\bs o\in \bs S}$ and the Palm distribution is obtained by conditioning on $\bs o\in \bs S$.
\end{example}

\begin{example}[Biasing]
	\label{ex:biasing-Palm}
	Let $b:\mstar\to\mathbb R^{\geq 0}$ be an initial bias and $\Phi:=b\bs \mu$ defined in~\eqref{eq:b mu}. Then, the intensity of $\Phi$ is $\omid{b(\bs o)}$ and the Palm distribution of $\Phi$ is just biasing the probability measure by $b(\bs o)$ (if $0<\omid{b(\bs o)}<\infty$). In particular, Theorem~\eqref{thm:Palm-unimodular} implies that, after biasing by $b$, $[\bs X, \bs o, b\bs \mu]$ would be unimodular, which is already shown in Example~\ref{ex:biasing}. This fact can be used to reduce some results to the case without initial biasing (e.g., Theorem~\ref{thm:randomwalk}).
\end{example}


\begin{example}[Extending a Unimodular Graph]
	\label{ex:extension}
	In many examples in the literature, a unimodular graph is constructed by adding new vertices and edges to a given unimodular graph, and then applying a biasing and a root-change (see~\cite{shift-coupling} for various examples in the literature). Here, we will show that all this examples are instances of the Palm distribution developed in this work. As an example, the unimodularization of the dual of a unimodular planar graph (Example~9.6 of~ \cite{processes}) is reduced to the Palm distribution of the counting measure on the set of faces (see below for the details). 
	\\
	We use the most general construction of this form mentioned in Section~5 of~\cite{shift-coupling}. In short, assume $[\bs G, \bs o]$ is a random rooted graph that is not unimodular, but the MTP holds on a subset $\bs S$ of $\bs G$ (
	see Definition~10 of~\cite{shift-coupling}). In other words, $\bs G$ is obtained by adding vertices and edges to $[\bs S, \bs o]$. Let $\bs \mu$ and $\Phi$ be the counting measures of $\bs S$ and $\bs G$ respectively. In the setting of this section, $[\bs G, \bs o, \bs \mu]$ is unimodular (but not $[\bs G, \bs o]$) and $\Phi$ can be regarded as an equivariant random measure on $[\bs G, \bs o, \bs \mu]$. Therefore, one can define the Palm distribution of $\Phi$ using Definition~\ref{def:Palm} (if the intensity is finite). Then, Theorem~\ref{thm:Palm-unimodular} implies that, under the Palm distribution, $[\bs G, \bs o]$ is a unimodular graph, as desired. In this setting, the general construction in Theorem~5 of~\cite{shift-coupling} (which covers the similar examples in the literature) is reduced to the explicit construction of the Palm distribution given in Theorem~\ref{thm:Palmconstruction}.
	\\
	For a continuum example, it can be seen that the product of a unimodular \rmm space $[\bs X, \bs o, \bs \mu]$ with an arbitrary compact measured metric space can be made unimodular (Example~\ref{ex:product}) and this is an instance of the Palm distribution. See the proof of Theorem~\ref{thm:amenable} in Subsection~\ref{subsec:amenable-proofs} for more explanation.
\end{example}

The final example is the Poisson point process as follows.

\begin{theorem}[Palm of Poisson]
	\label{thm:PalmofPoisson}
	Let $c>0$ and $\Phi$ be the Poisson point process on $\bs X$ with intensity measure $c\bs\mu$ (Example~\ref{ex:poisson}). Then, the Palm version has the same distribution as $[\bs X, \bs o, \bs \mu, \Phi+\delta_{\bs o}]$.	
\end{theorem}
Note that $\Phi+\delta_{\bs o}$ is different from $\Phi\cup\{\bs o\}$ if $\bs \mu$ has atoms. This result generalizes Mecke's theorem for the ordinary Poisson point process (see Theorem~3.3.5 of~\cite{bookScWe08}). We guess that this property characterizes the Poisson point process (which generalizes Slivnyak's theorem), but we do not discuss this here.

\begin{proof}
	By~\eqref{eq:intensity}, it is straightforward to show that $\lambda_{\Phi}=c$. We claim that for every measurable function $f:\mkdoublestar{2}\to\mathbb R^{\geq 0}$, one has
	\begin{equation}
		\label{eq:Poisson-Mecke}
		\omid{\sum_{y\in \Phi} f(\bs X, \bs o, y,\bs \mu, \Phi)} = c \omid{\int_{\bs X}f(\bs X, \bs o, y, \bs \mu, \Phi+\delta_{y})d\bs \mu(y)},
	\end{equation}
	which generalizes Mecke's theorem (see Theorem~3.2.5 of~\cite{bookScWe08}).
	Assuming this, Let $B\subseteq \mkstar{2}$ be any event. By~\eqref{eq:Palmconstruction}, \eqref{eq:Poisson-Mecke}, the MTP and the fact $h^-(\bs o)=1$ respectively,
	\begin{eqnarray*}
		\probPalm{\Phi}{B} &=& \frac 1 c \omid{\sum_{x\in \Phi} \identity{B}(\bs X, z,\bs \mu, \bs \Phi)h(\bs o, z)}\\
		&=& \omid{\int_{\bs X} \identity{B}(\bs X, z,\bs \mu, \bs \Phi+\delta_{z})h(\bs o, z)d\bs \mu(z)}\\
		&=& \omid{\int_{\bs X} \identity{B}(\bs X, \bs o,\bs \mu, \bs \Phi+\delta_{\bs o})h(z, \bs o)d\bs \mu(z)}\\
		&=& \myprob{[\bs X, \bs o, \bs \mu, \Phi+\delta_{\bs o}]\in B}
	\end{eqnarray*}
	and the theorem is proved. To prove~\eqref{eq:Poisson-Mecke}, for simplicity, we assume that $(\mathrm{supp}(\bs \mu), \bs \mu)$ has no automorphism a.s. (otherwise, one may add an extra randomness like an independent \textit{Poisson rain} to break the automorphisms). In this case, it is enough to prove this claim when the function $f$ is of the form $f=g(\bs X, \bs o, y, \bs \mu) \identity{\{\Phi(S)=k \}}$, where $k\in\mathbb N\cup\{0\}$, $S:=S(\bs X, \bs o, \bs \mu):=\{y\in \bs X: (\bs X, \bs o, y, \bs \mu)\in A\}$ and $A\subseteq\mdoublestar$ is measurable. The left hand side of~\eqref{eq:Poisson-Mecke} is
	\begin{eqnarray*}
		\omid{(g\Phi)(\bs X) \identity{\{\Phi(S)=k \}}} &=& \omid{(g\Phi)(\bs X\setminus S) \identity{\{\Phi(S)=k \}}} + \omid{(g\Phi)(S) \identity{\{\Phi(S)=k \}}}
		\\&=:&\omid{a_1}+\omid{a_2}.
	\end{eqnarray*}
	Conditioned on $[\bs X, \bs o, \bs \mu]$, the random variables $\Phi(S)$ and $(g\Phi)(\bs X\setminus S)$ are independent. In addition, by Mecke's formula, $\omidCond{(g\Phi)(\bs X\setminus S)}{[\bs X, \bs o, \bs \mu]}= c(g\bs \mu)(\bs X\setminus S)$. 
	Hence, 
	\[
		\omid{a_1}=c\omid{(g\bs \mu)(\bs X\setminus S) \identity{\{\Phi(S)=k \}}} = c \omid{\int_{\bs X\setminus S}f(\bs X, \bs o, y, \bs \mu, \Phi+\delta_{y})d\bs \mu(y)}.
	\]
	For the second term, conditioned on $[\bs X, \bs o, \bs \mu]$ and on $\Phi(S)=k$, $\restrict{\Phi}{S}$ is distributed as the set of $k$ random points with distribution proportional to $\restrict{\bs \mu}{S}$. So, $\omidCond{a_2}{[\bs X, \bs o, \bs \mu], \Phi(S)=k}=k(g\bs \mu)(S)/\bs \mu(S) \identity{\{\Phi(S)=k \}}$. 
	By the formula of the Poisson distribution with parameter $c\bs \mu(S)$, one can obtain that 
	$\omidCond{a_2}{[\bs X, \bs o, \bs \mu]} = c(g\bs \mu)(S) \identity{\{\Phi(S)=k-1\}}$. Thus,
	\[
		\omid{a_2}=c\omid{(g\bs \mu)(S) \identity{\{\Phi(S)=k-1\}}} = c \omid{\int_{S}f(\bs X, \bs o, y, \bs \mu, \Phi+\delta_{y})d\bs \mu(y)}.
	\]
	This proves~\eqref{eq:Poisson-Mecke} and the theorem is proved.
\end{proof}

\section{Some Applications of Palm Theory}
\label{sec:Palm-application}

In Subsections~\ref{subsec:disintegration} and~\ref{subsec:amenable-proofs}, we prove {equivariant disintegration} and the amenability theorem  (Lemma~\ref{lem:equivmeasure} and Theorem~\ref{thm:amenable}) using the Palm theory developed in Section~\ref{sec:Palm}. The proof is by constructing a countable Borel equivalence relation using the Poisson point process, and then, constructing an invariant measure using the Palm theory. Then, the results on Borel equivalence relations are used to prove the claims. Before that, we study the existence of \textit{balancing transport kernels} in Subsection~\ref{subsec:balancing}, which will be used in Subsection~\ref{subsec:amenable-proofs}.

\subsection{Balancing Transport Kernels}
\label{subsec:balancing}

Let $\Phi$ and $\Psi$ be jointly-equivariant random measures on a unimodular \rmm space $[\bs X, \bs o, \bs \mu]$ with positive and finite intensities. A \defstyle{balancing transport kernel} from $\Phi$ to $\Psi$ is a Markovian kernel on $\bs X$ (in the sense of Subsection~\ref{subsec:rerooting}, but might depend on $\Phi$ and $\Psi$ as well)  that transports $\Phi$ to $\Psi$; i.e., $\int_{\bs X} k(x,\cdot) d\Phi(x) = \Psi(\cdot), \forall x\in\bs X$ a.s. This extends the notion of \textit{invariant transport} for stationary random measures~\cite{ThLa09} and \textit{fair tessellations} for point processes~\cite{HoPe05}. On the Euclidean space, using the shift-coupling result of Thorisson~\cite{Th96}, a necessary and sufficient condition for the existence of invariant transports is provided in~\cite{ThLa09}. 
An analogous result is provided for unimodular graphs in~\cite{shift-coupling}. Here, we extend these results to unimodular \rmm spaces.


An event $A\subseteq \mkstar{3}$ is called \defstyle{invariant} if it does not depend on the root. Let $I$ denote the invariant sigma-field. The \defstyle{sample intensity} of $\Phi$ is defined by $\omidCond{h^+_{\Phi}(\bs o)}{I}$, where $h$ is the function in~\eqref{eq:intensity}.

\begin{theorem}
	\label{thm:balancing}
	If $[\bs X, \bs o, \bs \mu, \Phi, \Psi]$ is ergodic, then the existence of balancing transport kernels is equivalent to the condition that $\Phi$ and $\Psi$ have the same intensities. In the non-ergodic case, the condition is that the sample intensities of $\Phi$ and $\Psi$ are equal a.s.
\end{theorem}

The latter is equivalent to the condition that after conditioning to any invariant event, $\Phi$ and $\Psi$ have the same intensities. It is also equivalent to the condition that $\mathbb P_{\Phi}$ and $\mathbb P_{\Psi}$ agree on $I$. 

This theorem is an extension of Theorem~5.1 of~\cite{ThLa09}. To prove it, one might try to extend the \textit{shift-coupling} result of~\cite{Th96} (as done in~\cite{shift-coupling} for unimodular graphs), but we give another proof by extending the constructive proof of~\cite{stable}. The latter is an extension of~\cite{HoHoPe06} for point processes, which is an infinite version of the Gale-Shapley stable marriage algorithm.

First, note that the existence of a balancing transport kernel is equivalent to the existence of a \defstyle{balancing transport density}; i.e., a measurable function $K:\mkdoublestar{3}\to\mathbb R^{\geq}$ such that $\forall x,y\in X: \int_{\bs X}K(x,\cdot)d\Psi=\int_{\bs X}K(\cdot,y)d\Phi=1$ on an invariant event with full probability. To show this, if $k(x,\cdot)$ is absolutely continuous w.r.t. $\Psi$ for all $x$, it is enough to consider the Radon-Nikodym derivative equivariantly (as in Lemma~\ref{lem:rerooting-cont}) and to modify it on a null event. In the singular case, it is enough to \textit{smoothify} $k$ by composing it with another kernel that preserves $\Psi$ given by Lemma~\ref{lem:balancingkernel}.

\begin{proof}[Sketch of the Proof of Theorem~\ref{thm:balancing}]
	For simplicity, we assume ergodicity. Assume a balancing transport density $K$ exists. Under $\mathbb P_{\Phi}$, by letting $h:=K$ in~\eqref{eq:intensity}, one obtains that the intensity of $\Psi$ is equal to 1. So, the exchange formula in Corollary~\ref{cor:exchange} implies that $\lambda_{\Phi}=\lambda_{\Psi}$.
	For the converse, it is straightforward to extend Algorithm~4.4 of~\cite{stable} to construct a density $K$ (the \textit{stable constrained density}).
	Assuming $\lambda_{\Phi}=\lambda_{\Psi}$, one can use the Palm theory developed in Section~\ref{sec:Palm} and mimic the proof of Theorem~4.8 of~\cite{stable} to prove that $K$ is balancing. The proof is not short, but the extension is straightforward and is skipped for brevity.
\end{proof}

Theorem~\ref{thm:balancing} allows us to prove the following result, which generalizes a result of~\cite{Th96} (see the end of Section~1 of~\cite{Th96}).

\begin{theorem}
	\label{thm:shiftcouplingPalm}
	Let $\Phi$ be an equivariant random measure on a unimodular \rmm space $[\bs X, \bs o, \bs \mu]$. Then, the following are equivalent:
	\begin{enumerate}[(i)]
		\item \label{thm:shiftcouplingPalm-1} The Palm distribution of $\Phi$ is obtained by a random re-rooting of $[\bs X, \bs o, \bs \mu, \Phi]$. 
		\item \label{thm:shiftcouplingPalm-2} There exists a balancing transport kernel between $c\bs \mu$ and $\Phi$ for some $c$.
		\item \label{thm:shiftcouplingPalm-3} The sample intensity of $\Phi$ is constant.
	\end{enumerate}
	In particular, these conditions hold if $[\bs X, \bs o, \bs \mu, \Phi]$ is ergodic.
\end{theorem}

\begin{proof}
	The equivalence~\eqref{thm:shiftcouplingPalm-3}$\Leftrightarrow$\eqref{thm:shiftcouplingPalm-2} is implied by Theorem~\ref{thm:balancing}.

	\eqref{thm:shiftcouplingPalm-1}$\Rightarrow$\eqref{thm:shiftcouplingPalm-3}. 
	It is straightforward to deduce from~\eqref{thm:shiftcouplingPalm-1} that $\mathbb P_{\Phi}(A)=\myprob{A}$ for every invariant event $A$. This implies~\eqref{thm:shiftcouplingPalm-3}.
	
	\eqref{thm:shiftcouplingPalm-2}$\Rightarrow$\eqref{thm:shiftcouplingPalm-1}. Let $K$ be such a balancing transport density. In this case, Theorem~\ref{thm:balancing} for $h:=\frac 1 c K$ shows that the Palm version is obtained by re-rooting according to the kernel $h(\bs o, \cdot)\Phi$ (and there is no biasing since $h^+_{\Phi}=1$).
\end{proof}

\begin{example}
	\label{ex:balancing-Poisson}
	Assume $\Psi$ is the Poisson point process with intensity measure $\Phi$. Then, if $\Phi(\bs X)=\infty$ a.s., the ergodicity mentioned in Example~\ref{ex:poisson} implies that the conditions of Theorem~\ref{thm:balancing} are satisfied (if $[\bs X, \bs o, \bs \mu, \Phi]$ is not ergodic, use its ergodic decomposition). So, there exists a balancing transport density. In addition, in the case $\Phi=\bs \mu$, Theorem~\ref{thm:shiftcouplingPalm} and~\ref{thm:PalmofPoisson} imply that there exists a random re-rooting that is equivalent (in distribution) to adding a point at the origin to $\Psi$. This is an extension of \textit{extra head schemes}~\cite{HoPe05}.
\end{example}

\subsection{Proof of Equivariant Disintegration}
\label{subsec:disintegration}

In this section, we will prove Lemma~\ref{lem:equivmeasure} in the following steps.

\begin{lemma}
	\label{lem:disintegration1}
	The claim of Lemma~\ref{lem:equivmeasure} holds if $\bs \mu$ is a counting measures a.s. and every automorphism of $(\bs X, \bs \mu)$ fixes every atom of $\bs \mu$ a.s.
\end{lemma}
In this case, we will use results on countable Borel equivalence relations to deduce the claim from invariant disintegration for group actions.
\begin{proof}
	Let $A$ be the event that $\bs \mu$ is a counting measure and every automorphism of $(\bs X, \bs \mu)$ fixes every atom of $\bs \mu$.
	Consider the following countable equivalence relation on $\mstar$: Outside $A$ or if $o$ is not an atom of $\mu$, $[X,o,\mu]$ is equivalent to only itself. On $A$ and if $o_1$ and $o_2$ are atoms of $\mu$, let $[X,o_1,\mu]$ be equivalent to $[X,o_2,\mu]$. Then, this is a Borel equivalence relation on $\mstar$ and every equivalence class is countable. Therefore, by Theorem~1 of~\cite{FeMo77}, it is generated by the action of some countable group $G$ on $\mstar$. Note that on $A$, the map $y\mapsto [X,y,\mu]$ maps the atoms of $\mu$ bijectively to an equivalence class. So, on $A$, $G$ acts on the set of atoms of $\mu$. Let $g  y$ denote this action if $g\in G$ and $y$ is an atom of $\mu$ (outside $A$ or if $y$ is not an atom, one has $g  y = y$). This property allows us to extend the action of $G$ to $\mkstar{2}$ as follows (the assumption on automorphisms is essential for this goal): Given $[Y,p,\mu_1,\mu_2]\in\mkstar{2}$ and $g\in G$, let $g  [Y,p,\mu_1,\mu_2]:= [Y,g  p, \mu_1,\mu_2]$, where $g  p$ is defined by forgetting $\mu_2$ and using the above definition. 
	
	Let $Q$ and $\tilde Q$ be the distributions of $[\bs X, \bs o, \bs \mu]$ and $[\bs Y, \bs p, \bs \mu_1, \bs \mu_2]$ respectively.
	The two actions of $G$ on $\mstar$ and $\mkstar{2}$ are compatible with the projection $\pi:\mkstar{2}\to\mstar$, which is defined by forgetting the second measure. In addition, since every $g\in G$ acts bijectively on the set of atoms of $\bs \mu$ and $\bs \mu_1$, the distributions $Q$ and $\tilde Q$ are invariant under these actions (by Theorem~\ref{thm:rerooting}). Hence, we can use Kallenberg's invariant disintegration theorem (Theorem~3.5 of~\cite{disintegration}). This gives a disintegration kernel $k$ from $\mstar$ to $\mkstar{2}$ such that for all events $B\subseteq \mkstar{2}$,
	\begin{equation}
		\label{eq:disintegration1}
		\tilde Q(B) = \int_{\mstar} k([X,o,\mu],B) dQ([X,o,\mu])
	\end{equation}
	and $k(g  [X,o,\mu], g  B)=k([X,o,\mu],B)$ for all $g\in G$. This implies that $k([X,o,\mu],\cdot)$ is concentrated on $\{[X,o,\mu,\varphi]: \varphi\in M(X)\}$, where $M(X)$ is the set of boundedly-finite measures on $X$. Assuming $[X,o,\mu]\in A$, by applying a random element of the automorphism group of $(X,\mu)$ (which is a compact group since the atoms are fixed points) to $k([X,o,\mu],\cdot)$, one obtains an automorphism-invariant probability measure on $M(X)$, which we call it $k'(o, \cdot)$. Invariance under the action of $G$ implies that $k'(o_1,\cdot) = k'(o_2,\cdot)$ for all atoms $o_1$ and $o_2$ of $\mu$. So, for all $y\in X$, one can let $k'(y,\cdot):=k'(o,\cdot)$, where $o$ is an arbitrary atom of $\mu$. Now, by choosing $\Phi_{(X,y,\mu)}$ randomly with law $k'(y,\cdot)$, one obtains an equivariant random measure which satisfies all of the assumptions of Definition~\ref{def:equivmeasure}. In addition, \eqref{eq:disintegration1} implies that $[\bs X, \bs o,\bs \mu, \Phi]$ has law $\tilde Q$ and the claim is proved.
\end{proof}

\begin{lemma}
	\label{lem:disintegration2}
	The claim of Lemma~\ref{lem:equivmeasure} holds if there exists a factor point process $\bs S$ (depending only on $\bs X$ and $\bs \mu$) with finite intensity such that every automorphism of $(\bs X, \bs \mu)$ fixes every element of $\bs S$ a.s. and $\bs S$ is nonempty a.s.
\end{lemma}

\begin{proof}
	Let $Q$ and $\tilde Q$ be the distributions of $[\bs X, \bs o,  \bs \mu, \bs S]$ and $[\bs Y, \bs p, \bs \mu_1, \bs \mu_2, \bs S]$, as random elements of $\mkstar{2}$ and $\mkstar{3}$ respectively, where in the latter, $\bs S=\bs S(\bs Y, \bs \mu_1)$. The Palm distributions of $\bs S$ in these two unimodular objects give probability measures $Q_0$ and $\tilde Q_0$ respectively. By Theorem~\ref{thm:Palm-unimodular}, under $Q_0$, $[\bs X, \bs o, \bs \mu, \bs S]$ is unimodular with respect to the counting measure on $\bs S$ (Definition~\ref{def:unimodularw.r.t}) and the same holds for $[\bs Y, \bs p, \bs \mu_1, \bs \mu_2, \bs S]$. In addition, it is straightforward to show that, under $\tilde Q_0$, $[\bs Y, \bs p, \bs \mu_1, \bs S]$ has the same distribution as $Q_0$. Now, we can use Lemma~\ref{lem:disintegration1} to find an equivariant disintegration of $\tilde Q_0$ w.r.t. $Q_0$ (the same argument works if the counting measure of $\bs S$ is regarded as the base measure and $\bs \mu$ and $\bs \mu_1$ are thought of as decorations). This gives an equivariant random measure $\Phi$ (whose distribution depends on $(X,\mu,S)$) such that
	\[
		\tilde Q_0 (B) = \int \myprob{[X,o,\mu, \Phi, S]\in B} dQ_0 ([X,o,\mu, S])
	\]
	for all events $B$. By using Palm inversion to reconstruct $Q$ and $\tilde Q$ (Subsection~\ref{subsec:Palminversion}), it is straightforward to deduce that the above equation holds if $\tilde Q_0$ and $Q_0$ are replaced by $\tilde Q$ and $Q$ respectively. This shows that $\Phi$ is the desired equivariant random measure and the claim is proved.
\end{proof}

We are now ready to prove Lemma~\ref{lem:equivmeasure}.

\begin{proof}[Proof of Lemma~\ref{lem:equivmeasure}]
	First assume that $\bs \mu(\bs X)=\infty$ a.s.	
	We start by adding a \textit{marked Poisson point process} to ensure that the assumptions of Lemma~\ref{lem:disintegration2} hold. Let $\Psi=\Psi(\bs X, \bs \mu)$ be the Poisson point process on $\bs X$ with intensity measure $\bs \mu$ and equip the points of $\Psi$ with i.i.d. marks in $[0,1]$ chosen with the uniform distribution. By a suitable extension of the GHP metric discussed in Subsection~\ref{subsec:extra}, one can think of $[\bs X, \bs o, \bs \mu, \Psi]$ and $[\bs Y, \bs p, \bs \mu_1, \bs \mu_2, \Psi]$ as unimodular \rmm spaces equipped with additional structures, where in the latter, $\Psi=\Psi(\bs Y, \bs \mu_1)$. 
	In the former, the probability space is the set $\mstar'$ of tuples $[X,o,\mu,\psi]$, where $(X,o,\mu)$ is a \rmm space and $\psi$ is a \textit{marked measure} on $X$; i.e., a boundedly-finite measure on $X\times [0,1]$.
	
	Note that the support of $\Psi$ is a factor point process (depending on $\bs \mu$ and $\Psi$) such that every point of the subset is fixed under every automorphism of $(\bs X, \bs \mu, \Psi)$ a.s. (due to the i.i.d. marks). In addition, the support is not empty a.s. since $\bs \mu(\bs X)=\infty$ a.s. So, we can use an argument similar to Lemma~\ref{lem:disintegration2} to obtain an equivariant random measure $\Phi$ such that $[\bs X, \bs o, \bs \mu, \Phi, \Psi]$ has the same distribution as $[\bs Y, \bs p, \bs \mu_1, \bs \mu_2,\Psi]$. 
	
	Note that $\Phi=\Phi(X,\mu,\psi)$ is defined for deterministic spaces $(X,\mu,\psi)$ and not for the spaces of the form $(X,\mu)$.
	To resolve this issue, one can take expectation w.r.t. $\Psi$ as follows. Let $(X,o,\mu)$ be a realization of $[\bs X, \bs o, \bs \mu]$. Given a realization $\psi$ of the marked Poisson point process on $(X,\mu)$, let $\alpha = \alpha_{(X,\mu, \psi)}$ be the distribution of $\Phi=\Phi(X,\mu,\psi)$. Here, $\alpha$ is a probability measure on $M(X)$. Let $\beta_{(X,\mu)}:=\omid{\alpha_{(X,\mu,\Psi)}}$, where the expectation is w.r.t. to the randomness of $\Psi$. It is straightforward to see that $\beta$ is invariant under the automorphisms of $(X,\mu)$. Now, by choosing a random measure $\Phi'$ on $X$ with distribution $\beta_{(X,\mu)}$, one obtains an equivariant random measure such that $[\bs X, \bs o, \bs \mu, \Phi']$ has the same distribution as $[\bs Y, \bs p, \bs \mu_1, \bs \mu_2]$ as desired. So the claim is proved.
	
	Finally, assume $\bs \mu(\bs X)<\infty$ with positive probability. Conditioned on the event $\bs \mu(\bs X)=\infty$, the above arguments provide the equivariant disintegration. Conditioned on the event $\bs \mu(\bs X)<\infty$, the situation is easier. Under this conditioning, $[\bs X, \bs \mu]$ and $[\bs Y, \bs \mu_1, \bs \mu_2]$ make sense as random non-rooted measured metric spaces and the corresponding probability spaces are Polish (under suitable topologies which are skipped here). So, it is enough to consider the regular conditional distribution of $[\bs Y, \bs \mu_1, \bs \mu_2]$ w.r.t. $[\bs Y, \bs \mu_1]$, and then apply a random automorphism of $(\bs Y, \bs \mu_1)$ to obtain an equivariant random measure. This automatically does not depend on the root and has the desired properties.
\end{proof}

\subsection{Proof of the Amenability Theorem}
\label{subsec:amenable-proofs}

In this subsection, we prove that the different notions of amenability defined in Subsection~\ref{subsec:amenable} are equivalent (Theorem~\ref{thm:amenable}). 

Let $[\bs X, \bs o, \bs \mu]$ be a unimodular \rmm space. First, we assume that $\bs \mu$ does not have atoms a.s. and $\bs \mu(\bs X)=\infty$ a.s. As in Subsection~\ref{subsec:disintegration}, let $\Phi$ be the Poisson point process on $\bs X$ with intensity measure $\bs \mu$ (Example~\ref{ex:poisson}), equipped with i.i.d. marks in $[0,1]$ chosen with the uniform distribution. Then, $\Phi$ is infinite, there are no multiple points and $\Phi$ has no nontrivial automorphism a.s.
Let $\mathbb P_{\Phi}$ be the Palm distribution of $\Phi$. By Theorem~\ref{thm:Palm-unimodular}, under $\mathbb P_{\Phi}$, $[\bs X, \bs o, \bs \mu, \Phi]$ is unimodular w.r.t. $\Phi$.

The above definitions give a countable equivalence relation as follows. The probability measure $\mathbb P_{\Phi}$ is concentrated on the subset $\mstar''\subseteq\mstar'$ of tuples $[X,o,\mu,\varphi]$ in which $\mu$ has no atom, $\mu(X)=\infty$, $\varphi$ is the counting measure on an infinite discrete subset of $X$, $o\in \varphi$ and the marks of the points of $\varphi$ are distinct. Let $R$ be the equivalence relation on $\mstar'$ defined as follows: Outside $\mstar''$, everything is equivalent to only itself. If $[X,o,\mu,\varphi]\in \mstar''$, then it is equivalent to $[X,y,\mu,\varphi]$ for all $y\in \varphi$. Then, $R$ is a countable equivalence relation on the Polish space $\mstar'$. In addition, unimodularity of $\mathbb P_{\Phi}$ is equivalent to invariance of $\mathbb P_{\Phi}$ under $R$ (see Subsection~\ref{subsec:Borel} and note that having no automorphism is important for this equivalence).

\begin{lemma}
	\label{lem:amenable}
	If $\bs \mu$ has infinite total mass and no atom a.s., then each of the conditions of the existence of local means, the existence of approximate means, and hyperfiniteness is equivalent to the analogous condition for the countable measured equivalence relation $(\mstar', R, \mathbb P_{\Phi})$ defined above.
\end{lemma}

\begin{proof}
	An essential ingredient of the proof is the Voronoi tessellation: For every $[X,o,\mu,\varphi]\in\mstar''$, let $\tau(o)$ be the closest point of $\varphi$ to $o$. If the closest point is not unique, choose the one with the smallest mark. For $y\in \varphi$, $\tau^{-1}(y)$ is the \textit{Voronoi cell} of $y$. Another ingredient is the balancing transport density constructed in Subsection~\ref{subsec:balancing} (Example~\ref{ex:balancing-Poisson}). This gives an equivariant function $k(x,y)$ for $x\in \bs X$ and $y\in \Phi$ such that $\sum_{z\in \Phi} k(x,z)=1$ and $\int k(z, y)d\bs \mu(z)=1$ a.s. (for all $x$ and $y$). 
	From now on, we always assume $[X,o,\mu,\varphi]\in \mstar''$ and the above balancing property of $k$ holds. In the next paragraphs, we prove each equivalence claimed in the lemma. 
	
	\textbf{Local Mean}.  Given $f\in L^{\infty}(X,\mu)$, define $f'\in L^{\infty}(\varphi)$ by $f'(y):=\int f(z) k(z,y) d\mu(z)$ (the balancing property of $k$ is important here). Conversely, given $g\in L^{\infty}(\varphi)$, define $g'\in L^{\infty}(X,\mu)$ by $g'(x):=\sum_{z\in \varphi} g(z)k(x,z)$. Using this, every mean on $L^{\infty}(\varphi)$ gives a mean on $L^{\infty}(X,\mu)$ and vice versa. This implies the claim.
	
	\textbf{Approximate mean}. Assume $(\lambda_n)_n$ is an approximate mean. For $x,y\in \varphi$, define $\Lambda_n(x,y):=\int \lambda_n(x,z)k(z,y)d\mu(z)$. This gives approximate means for $\varphi$. Equivalently, $(\Lambda_n)_n$ satisfies the condition (AI) of~\cite{Ka97} (it is important that $\varphi$ has no nontrivial automorphism). Conversely, given approximate means $\Lambda_n$ for $(R,\mathbb P_{\Phi})$, for $x,y\in X$, define $\lambda_n(x,y):= \sum_z\sum_t k(x,z)\Lambda_n(z,t)k(t,y)$. It is left to the reader to prove that $\lambda_n$ is an approximate mean.
	
	\textbf{Hyperfiniteness}. 
	In Lemma~\ref{lem:hyperfinite}, it is proved that~\eqref{HF2} $\Rightarrow$ \eqref{HF3} $\Rightarrow$ \eqref{HF1}. Assume $(\Pi_n)_n$ satisfies~\eqref{HF1} and, conditioned to $[\bs X, \bs o, \bs \mu]$, the extra randomness in the definition of $(\Pi_n)_n$ is independent of $\Phi$ (use Example~\ref{ex:independentcoupling}). Hence, since each element of $\Pi_n$ has finite mass, it has also finitely many points in the Poisson point process a.s. So, $(\Pi_n)_n$ induces equivariant nested finite partitions of $\Phi$. If there is no extra randomness (i.e., $(\Pi_n)_n$ is a factor), it also induces a nested sequence of Borel equivalence sub-relations of $R$ and $R$ is hyperfinite. If not, one can enlarge $\mstar'$ according to the extra randomness of the partitions, but this does not affect the hyperfiniteness of the Borel equivalence relation (using Theorem~1 of~\cite{Ka97}, find a local  and then take its expectation w.r.t. to the extra randomness to find a local mean for $(R,\mathbb P_{\Phi})$). 
	\\
	Finally, assume $(R,\mathbb P_{\Phi})$ is hyperfinite. This gives equivariant nested finite partitions $\Pi_n$ of $\varphi$ such that $\forall y\in \varphi, \forall  r<\infty, \exists n: B_r(y)\cap \varphi\subseteq \Pi_n(y)$. Define the partition $\Pi'_n$ of $X$ as follows: $x\in \Pi'_n(y)$ if and only if $\tau(x)\in \Pi_n(\tau(y))$. The sequence $(\Pi'_n)_n$ satisfies~\eqref{HF2} since $\tau(B_r(\bs o))$ is a finite set a.s. for every $r$ (since $\tau(B_r(\bs o))\subseteq B_{2r+s}(\bs o)$, where $s:=d(\bs o, \tau(\bs o))$, which is straightforward to verify). This finishes the proof.
\end{proof}

\begin{proof}[Proof of Theorem~\ref{thm:amenable}]
	The Folner conditions are already treated in Lemma~\ref{lem:hyperfinite}. Here, we prove the equivalence of the rest of the conditions.
	On the event $\bs \mu(\bs X)<\infty$, all of the conditions hold. So it is enough to assume $\bs \mu(\bs X)=\infty$ a.s. 
	If $\bs \mu$ has no atoms a.s., then Lemma~\ref{lem:amenable} implies that the different notions of amenability of $[\bs X, \bs o, \bs \mu]$ are equivalent to those for $(\mstar', R, \mathbb P_{\Phi})$, defined above. So, the equivalence of the conditions is implied by Theorem~1 of~\cite{Ka97}.
	
	Finally, if $\bs \mu$ is allowed to have atoms, we multiply $\bs X$ by $[0,1]$ to destroy the atoms as follows.
	Let $\bs X':=\bs X\times [0,1]$ equipped with the sum metric $d((x_1,t_1),(x_2,t_2)):=d(x_1,x_2)+\norm{t_1-t_2}$. 
	Let $\bs \mu':=\bs \mu\times \mathrm{Leb}$. Then, by choosing $\bs o'$ in $\{\bs o\}\times [0,1]$ uniformly, $[\bs X', \bs o', \bs \mu', \bs X\times \{0\}]$ is unimodular, where $\bs X\times \{0\}$ is kept as a distinguished closed subset of $\bs X'$ (Example~\ref{ex:product}; this is in fact the Palm version of $\bs \mu'$ by Theorem~\ref{thm:Palmconstruction}). Note that the product structure can be recovered from $(\bs X', \bs X\times \{0\})$. It is left to the reader to show that the different notions of amenability for $[\bs X, \bs o, \bs \mu]$ are equivalent to those for $[\bs X', \bs o', \bs \mu', \bs X\times \{0\}]$. Since $\bs \mu'$ has no atoms, the first part of the proof implies the claim.
\end{proof}

\section*{Acknowledgments}
This work was supported by the ERC NEMO grant, under the European Union's Horizon 2020 research and innovation programme, grant agreement number 788851 to INRIA. A major part of the work was done when the author was affiliated with IPM. The research was in part supported by a grant from IPM (No. 98490118).

\bibliography{bib} 
\bibliographystyle{plain}

\end{document}

%% file: commands.tex
\newcommand{\defstyle}[1]{\textbf{#1}}
\newcommand{\myprob}[1]{\mathbb P \left[ #1 \right]}
\newcommand{\probPalm}[2]{\mathbb P_{#1} \left[ #2 \right]}

\newcommand{\omid}[1]{\mathbb E \left[ #1 \right]}
\newcommand{\omidPalm}[2]{\mathbb E_{#1} \left[ #2 \right]}
\newcommand{\omidCond}[2]{\mathbb E \left[ #1 \left| #2 \right. \right]}

\newcommand{\norm}[1]{\left| #1 \right|}

\newcommand{\identity}[1]{1_{#1}}
\newcommand{\bs}[1]{\boldsymbol{#1}}

\newcommand{\del}[1]{}
\newcommand{\mar}[1]{\marginpar{\scriptsize  #1}}
\newcommand{\unwritten}[1]{}

\newcommand{\rmm}{rmm }

\newcommand{\supp}{\mathrm{supp}}

\newcommand{\oball}[2]{B_{#1}(#2)}
\newcommand{\cball}[2]{\overline{B}_{#1}(#2)}

\newcommand{\prokhorov}{d_P}

\newcommand{\restrict}[2]{{
		\left.\kern-\nulldelimiterspace 
		#1 
		\vphantom{\big|} 
		\right|_{#2} 
}}

\newcommand{\mstar}{\mathcal M_*}
\newcommand{\mdoublestar}{\mathcal M_{**}}
\newcommand{\mkstar}[1]{\mathcal M^{#1}_*}
\newcommand{\mkdoublestar}[1]{\mathcal M^{#1}_{**}}

\theoremstyle{theorem}
\newtheorem{theorem}{Theorem}[section]
\newtheorem{lemma}[theorem]{Lemma}
\newtheorem{proposition}[theorem]{Proposition}
\newtheorem{corollary}[theorem]{Corollary}

\newtheorem{problem}[theorem]{Problem}

\theoremstyle{definition}
\newtheorem{definition}[theorem]{Definition}

\newtheorem{example}[theorem]{Example}
\theoremstyle{definition}
\newtheorem{remark}[theorem]{Remark}
\newtheorem{convention}[theorem]{Convention}

\theoremstyle{theorem}

\numberwithin{equation}{section}

\makeatletter
\let\orgdescriptionlabel\descriptionlabel
\renewcommand*{\descriptionlabel}[1]{%
	\let\orglabel\label
	\let\label\@gobble
	\phantomsection
	\edef\@currentlabel{#1}%
	\let\label\orglabel
	\orgdescriptionlabel{#1}%
}